\documentclass[preprint]{elsarticle}

\usepackage{graphicx}
\usepackage{amsmath,amsthm}
\usepackage{graphicx}
\usepackage{amsfonts}
\usepackage{amssymb}
\usepackage{subfigure}
\usepackage{url}
\usepackage{enumerate}
\usepackage[mathscr]{eucal}
\usepackage[utf8]{inputenc}

\newtheorem{theorem}{Theorem}[section]

\newtheorem{corollary}[theorem]{Corollary}

\newtheorem{lemma}[theorem]{Lemma}

\newtheorem{remark}[theorem]{Remark}

\numberwithin{equation}{section}
\newtheorem*{theorem*}{Theorem}

\DeclareMathOperator*{\esssup}{ess\,sup}
\usepackage{lineno}

\journal{Physica D}

\begin{document}

\begin{frontmatter}
\title{Optical beam shaping and diffraction free waves: a variational approach}
\author[az]{John A. Gemmer\corref{cor1}}
\ead{john\_gemmer@brown.edu}
\author[bz,cz]{Shankar C. Venkataramani}
\author[dz]{Charles G. Durfee}
\author[bz,cz,ez]{Jerome V. Moloney}

\address[az]{Division of Applied Mathematics, Brown University, Providence, RI 02912, U.S.A.}
\address[bz]{Arizona Center for Mathematical Sciences, Department of Mathematics, University of Arizona, Tucson, AZ 85721, U.S.A.}
\address[cz]{Department of Mathematics, University of Arizona, Tucson, AZ 85721, U.S.A.}
\address[dz]{Department of Physics, Colorado School of Mines, Golden, CO 80401, U.S.A.}
\address[ez]{College of Optical Sciences, University of Arizona, Tucson, AZ 85721, U.S.A.}
\cortext[cor1]{Principal corresponding author}

\begin{keyword}
Beam shaping; Localized waves; Phase retrieval; Paraxial wave equation; Fresnel approximation
\end{keyword}

\begin{abstract} 
We investigate the problem of shaping radially symmetric annular beams into desired intensity patterns along the optical axis. Within the Fresnel approximation, we show that this problem can be expressed in a variational form equivalent to the one arising in phase retrieval. Using the uncertainty principle we prove various rigorous lower bounds on the functional; these lower bounds  estimate the $L^2$ error for the beam shaping problem in terms of the design parameters. We also use the method of stationary phase to construct a natural \emph{ansatz} for a minimizer in the short wavelength limit.  We illustrate the implications of our results by applying the method of stationary phase coupled with the Gerchberg-Saxton algorithm to beam shaping problems arising in the remote delivery of beams and pulses. 
\end{abstract}

\end{frontmatter}


%% % % % % % % % % % % % % % % % % % % % % %
%
%		Introduction
%
% % % % % % % % % % % % % % % % % % % % % %
\section{Introduction}

%%%%%%%%%%%%%%%%%%%
%
%		Motivation
%
%%%%%%%%%%%%%%%%%%%%
\subsection{Physical description and motivation of the problem}
In many applications it is desirable to shape a beam or pulse of light so that it has specific  properties along the optical axis. In particular, for applications in microscopic imaging \cite{Planchon:2011tp}, optical tweezers \cite{Arlt:2001up}, laser micro-machining \cite{Courvoisier:2012ik}, dressing of optical filaments \cite{mills2013dressed}, filament formation \cite{Polynkin:2008vm} and long-range laser ablation \cite{stelmaszczyk2004long, rohwetter2005filament}, to name a few, it is important to have a well controlled beam with a nearly uniform intensity along the optical axis. However, due to their wave nature localized packets of light will broaden spatially through diffraction. For example, Gaussian beams of width $W_0$ and wavenumber $k$ double in spatial extent over the Rayleigh range $z_R\sim W_0^2k/2$ \cite{rayleigh1881xxix}. 

%The dispersive material properties of the medium the light is propagating in will also lead to temporal broadening.  For example,  pulses with a Gaussian temporal width $\tau_0$ propagating in a medium with a group velocity dispersion $\gamma$ the temporal width doubles at the dispersion distance $z_D\sim \tau_0^2/\gamma$ \cite{siegman1986university}. 

In a linear, isotropic medium the electric field of a beam of light polarized in the $\hat{\mathbf{e}}$ direction can be modeled by a wave packet of the form $\mathbf{E}(\mathbf{x}_{\perp}, z)=E(x,y,z)\exp\left(i k z-i\omega_0t\right)\hat{\mathbf{e}}$, where $z\in \mathbb{R}^+$ denotes the spatial coordinate along the optical axis, $x,y\in \mathbb{R}$ are the Cartesian coordinates transverse to the optical axis, $t\in \mathbb{R}^+$ is time and $\omega_0$ the carrier angular frequency of the wave. Within the slowly varying \emph{ansatz} on $E$, the propagation of a beam with initial data $E(x,y,0)=E_0f(x,y)$ at the $z=0$ plane is described by the paraxial wave equation:
\begin{equation}\label{intro:GoverningEquation}
\begin{cases}
\frac{\partial E}{\partial z}=\frac{i}{2k}\Delta_{{\perp}}E,\\
E(x,y,0)=E_0f(x,y)
\end{cases}
\end{equation}
where $\Delta_{\perp}$ is the transverse Laplacian defined by $\Delta_{\perp}=\partial^2_x+\partial^2_y$ \cite{newell1992nonlinear}. 
Exact solutions to Equation (\ref{intro:GoverningEquation}) can be expressed in terms of an integral transform of the initial profile $f(x,y)$; see \ref{Appendix:Solution}. This integral transform is in fact equal to the Fresnel diffraction integral \cite{grella1982fresnel}.

Constructing solutions to (\ref{intro:GoverningEquation}) in free space that resist diffraction over length scales much larger than the Rayleigh range has attracted considerable interest within the optics community. Diffraction free beams are solutions to Equation (\ref{intro:GoverningEquation}) satisfying the condition that for all $z\in \mathbb{R}^+$, $|E(x,y,0)|= |E(x,y,z)|$. If we make the \emph{ansatz} $E(x,y,z)=A(x,y)\exp(i k^{\prime} z)$ then diffraction free beams can be found by solving the two dimensional Helmholtz equation. Indeed, plane waves, Bessel beams \cite{durnin1987exact, durnin1987diffraction}, Mathieu beams  \cite{gutierrez2000alternative} and parabolic beams \cite{bandres2004parabolic} are examples of such solutions in various orthogonal coordinate systems. 

%Polychromatic superpositions of Bessel beams have also been constructed that in addition to being diffraction free remain temporally localized  \cite{lu1992nondiffracting, porras2004localized}. These pulses are time dependent solutions to (\ref{intro:GoverningEquation}) for which the spectrum is a Dirac delta function supported on a curve in $\mathbf{k}-\omega$ space \cite{porras2004localized}. The exact form of this curve is selected so that the pulse has a specific group and phase velocity along the optical axis. In particular, for normal dispersion $(\gamma>0)$ the curves are hyperbolas while for anomalous dispersion $(\gamma<0)$ the curves are ellipses and hence these solutions to (\ref{intro:GoverningEquation}) have been called X and O waves respectively \cite{lu1992nondiffracting, porras2004localized}.

Note, however, that the diffraction free beams described above are not realizable in practice since they carry infinite energy. One technique for obtaining finite energy approximations of such solutions is apodization at the input plane $z=0$. That is, if $E(x,y,z)$ solves (\ref{intro:GoverningEquation}) then apodized approximations to $E(x,y,z)$ are found by solving (\ref{intro:GoverningEquation}) with the initial data $E(x,y,0)f(x,y)$ where $f(x,y)$ is a function of finite extent. In \cite{Gori:1987jk} so called ``Bessel-Gauss'' beams were created by apodizing Bessel beams with a Gaussian function and a general framework for constructing Gaussian apodized solutions to (\ref{intro:GoverningEquation}) was developed in \cite{gutierrez2005helmholtz} and \cite{Graf:2012wy}. 

%Specifically, we consider starting with a Gaussian ring-beam and adjusting its wavefront to produce a desired axial intensity profile. Such a ring beam can be formed, for example, by passing a Gaussian beam through a diffraction grating with concentric grooves [20].

Bessel-Gauss beams have received considerable interest because they can be produced experimentally, for example by a diffraction grating with concentric grooves \cite{niggl1997properties}. Alternatively, a Gaussian beam passing through an axicon lens produces a ring beam with a truncated Gaussian cross-section that results from the division of the beam around the axicon tip \cite{mcgloin2005bessel}. The effect of these optical elements is to create a linear superposition of Gaussian beamlets whose individual wave vectors $k=k_{\perp}\sin(\theta)$ lie on a cone of aperture angle $\theta$. The mutual interference of these beamlets generates the Bessel pattern and due to influx of energy supplied by the oscillatory wings of the Bessel function to the central core, these waves are robust to nonlinear losses and physical barriers \cite{bouchal1998self}. However, for a truncation of width $W_A$ the Bessel-Gauss beam only accurately approximates the ideal Bessel beam in a ``Bessel-zone'' of width $z_{BZ}\sim W_{A}/\tan(\theta)$ near the tip of the axicon \cite{graf2012asymptotic} and then transitions to a ring-beam in the far-field \cite{Gori:1987jk, Bagini}. 

It was also shown by Bagini et. al. \cite{Bagini} that Bessel-Gauss beams can be created remotely by focusing ring beams. Specifically, in cylindrical coordinates $(r,\theta,z)$  Bessel-Gauss beams can be generated remotely by solving Equation (\ref{intro:GoverningEquation}) with the initial data 
\begin{equation}
E(r,0)=E_0\exp\left(-\frac{r^2+a^2}{W_0^2}\right)J_0\left(i\frac{2ar}{z_d}+\beta r\right),
\end{equation}
 where $E_0$ is the peak value of the electric field, $a$ is the radius of the ring beam, $W_0$ is the characteristic width of the ring, $z_d$ is the focal distance, $J_0$ is the zero order Bessel function of the first kind and $\beta$ is a parameter related to the focusing of the beam \cite{Bagini}. Note that for this initial data the argument of the Bessel function is complex valued and in particular if $\beta=0$ -- no focusing -- it reduces to a modified Bessel function of the first kind. In Figure \ref{fig:GenBesselBeams} we provide a contour plot of the $r-z$ intensity profile for such a ``generalized'' Bessel-Gauss beams and select four radial cross sections to highlight the transition from a ring-beam into a Bessel-Gauss beam and back to a ring beam. As can be seen from Figure \ref{fig:GenBesselBeams}, the intensity profile along the optical axis ($r=0$) is not uniform in $z$ within the Bessel-zone and is in fact given by a Lorentzian-Gaussian type function \cite{Gori:1987jk, Bagini}.

\begin{figure}[ht]
                \includegraphics[width=\textwidth]{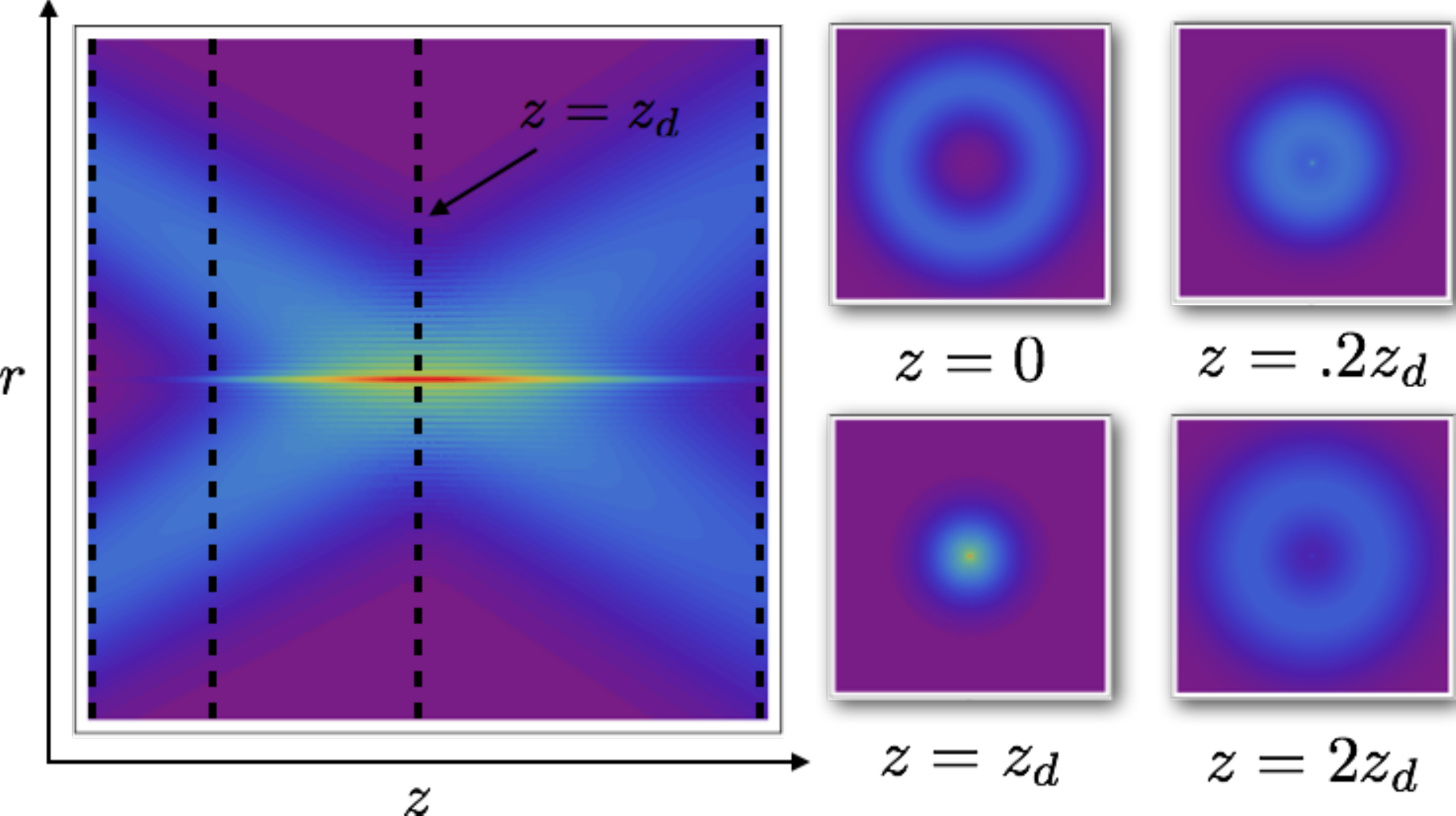}
                \caption{ The large image on the left is a contour plot of the square root of the intensity profile for a representative of the generalized Bessel-Gauss beams discovered by Bagini et. al. \cite{Bagini}. The four smaller plots on the right are cross sections of the beam transverse to the optical-axis corresponding to the vertical dashed lines.  The figure illustrates the transition of the ring beam into a Bessel-Gauss beam and back into a ring beam in the far field along the optical axis. }\label{fig:GenBesselBeams}
\end{figure}

The non-uniformity of the on-axis intensity limits the utility of Bessel-beams in some applications. In this paper we explore the problem of applying additional phase shaping to the ring beam in order to create beams with specific properties along the optical axis. Specifically, we consider the problem of constructing a phase function $\phi(r)$ at the input plane $z=0$ that focuses a ring beam intensity profile $E_0f(r)$ of radius $r_0$ and width $W_0$ into a target intensity profile $E_TF_T(z)$ of width $W_T$ along the optical axis centered at a target distance $z_d$. In Figure \ref{fig:Besselsubfig2} we illustrate the geometry and notation for the problem we are considering. The constants $E_0,\, E_T>0$ denote peak values of the electric field and $f(r)$, $F_T$ are smooth, non negative compactly supported functions with a maximum value of $1$ supported on the intervals $S_f$ and $S_{F_T}$  defined by
\begin{equation}
S_f=\left\{ r\in \mathbb{R}^+: |r-r_0|\leq \frac{W_0}{2} \right\} \text{ and }
S_{F_T}=\left \{z\in \mathbb{R}^+ : |z-z_d| \leq \frac{W_T}{2} \right\},
\end{equation}
respectively. Simultaneous amplitude and phase shaping of the beam will achieve the task and indeed this problem was studied in \cite{vcivzmar2009tunable}. However, in practice this is achieved through holographic techniques which are difficult to implement and have low power efficiency \cite{vcivzmar2009tunable}. 

In Figure \ref{fig:Besselsubfig1} we illustrate one possible method for how phase shaping of ring beams can be achieved. First, a ring beam of radius $r_0$ and width $W_0$ is created by passing a collimated beam through an axicon or a concentric diffraction grating. The phase function is then applied, perhaps through a lens with a radially dependent thickness, which focuses the beam into the desired profile. 

\begin{figure}
        \centering
          \subfigure[][]{
                \includegraphics[width=.9\textwidth]{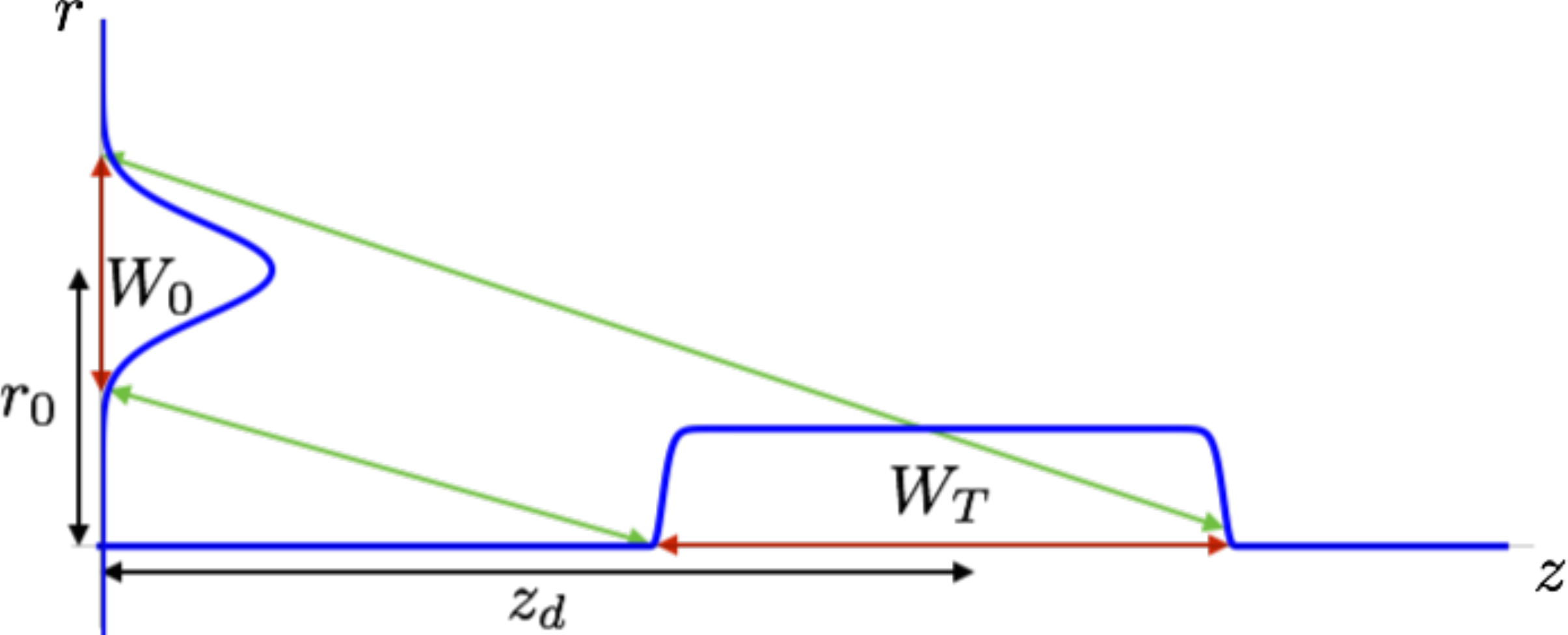}
                \label{fig:Besselsubfig2}}
        \subfigure[][]{
                \includegraphics[width=.9\textwidth]{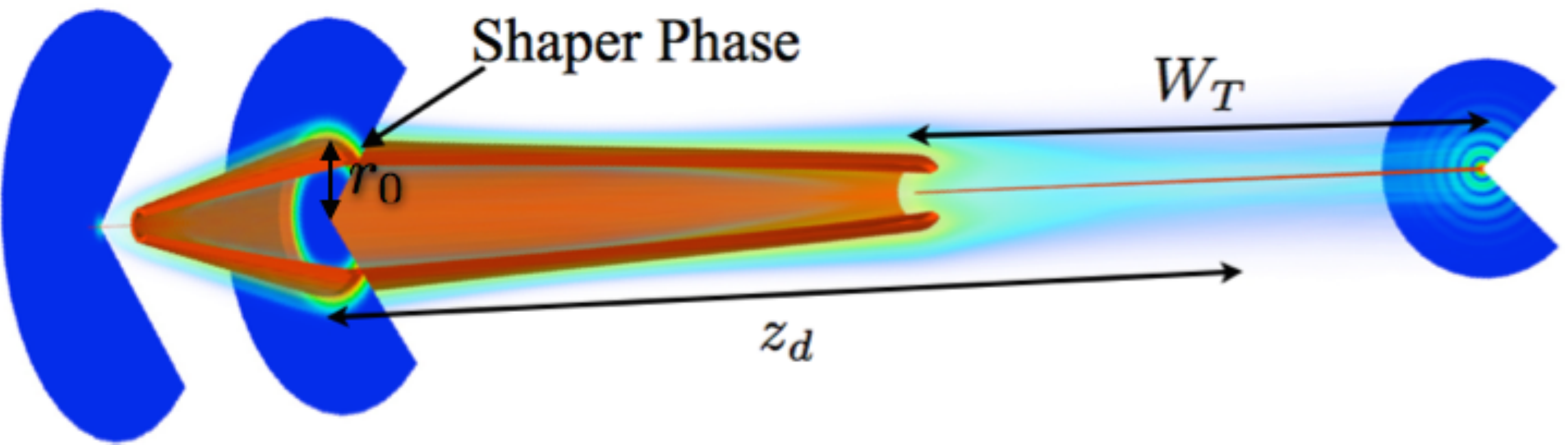}
	\label{fig:Besselsubfig1}}
      
                \caption{\subref{fig:Besselsubfig2} Illustration of the geometry and notation we are using for this problem. \subref{fig:Besselsubfig1} The optical setup for the problem we are studying. A beam is first mapped into a cylindrical beam $E_0f(r)$ of radius $r_0$ and width $W_0$. An optical element is then applied to the cylindrical beam that radially modifies the phase of the cylindrical beam so that the beam is focused a distance $z_d$ into a target intensity profile $E_TF_T(z)$ of width $W_T$. }\label{fig:BesselBeams}
\end{figure}

\subsection{Mathematical formulation of the problem}
The effect of adding a shaper phase $\phi(r)$ can be realized by solving Equation (\ref{intro:GoverningEquation}) with the initial data $E(r,0)=E_0f(r)\exp(i \phi(r))$. The exact solutions to Equation (\ref{intro:GoverningEquation}) can be expressed in terms of a Hankel transform and the field along the optical axis is given by the following relationship:
\begin{equation}\label{Intro:OpticalAxis}
E(0,z)=-\frac{i k}{z}\int_0^{\infty}E_0f(\rho)\exp\left(i \phi\left(\rho\right)\right)\exp\left(\frac{ik\rho^2}{2z}\right)\rho d\rho;
\end{equation}
(see \ref{Appendix:Solution}). In this paper we consider the problem of optimizing the on-axis profile of the beam in the sense that we  are interested in minimizing the  $L^2$ norm:
\begin{equation}
\|E_TF_T(z)- |E(0,z)|\|_{L^2}^2=\int_{-\infty}^{\infty}\left( E_TF_T(z)-|E(0,z)|\right)^2\,dz
\end{equation}
over the space of measurable phase functions $\phi$. Note that while the optimal axis is defined by the half-line $z>0$ we have defined the $L^2$ norm over the entire real axis. We make this definition as it allows us to use the $L^2$ isometry between a function and its Fourier transform. Moreover, this definition of the functional penalizes the beam wasting light outside the support of the target intensity distribution.

Since we are assuming that $f$ is compactly supported away from zero it follows from the change of variables $s=\rho^2$ and $\Omega=k (2z)^{-1}$ that
\begin{equation}\label{Form:FourierTrans}
E\left(0,\frac{k}{2\Omega}\right)=-i \Omega \mathcal{F}\left [f\left(\sqrt{s}\right) e^{i \phi\left(\sqrt{s}\right)} \right](\Omega),
\end{equation}
where $\mathcal{F}$ denotes the Fourier transform from $s$-space into $\Omega$-space. The optimal design problem can then abstractly be posed as finding a measurable phase function $\varphi$ that minimizes the functional $I[\varphi]$ defined by
\begin{equation}
I[\varphi]=\left\| G(\Omega)- |\mathcal{F} \left[g(s) \exp(i \varphi(s))\right](\Omega)\right|\|_{L^2},
\end{equation}
where $G$ and $g$ are positive compactly supported functions. This abstract formulation can be identified with the specific application we are interested in by setting
\begin{equation}
\begin{array}{ccc}
G(\Omega)=\Omega^{-1}E_TF_T\left(\frac{k}{2\Omega}\right), & g(s)=E_0f(\sqrt{s}), & \varphi(s)=\phi(\sqrt{s}).
\end{array}
\end{equation}
In the two-dimensional setting similar functionals have been studied within the context of shaping beams into desired intensity patterns in the focal plane of a lens \cite{dobson1992phase, kotlyar1998iterative, ripoll2004review, pasienski2008high}.

\subsection{Connection with phase retrieval}
The variational problem we are considering is closely related to the problem of phase retrieval from two intensity measurements, i.e. the problem of determining the complex argument of a function given both knowledge of the modulus of a function and the modulus of its Fourier transform. Phase retrieval arises in applications in various fields ranging from astronomy \cite{Gonsalves:1982ud, fienup1987phase}, ultrashort pulses \cite{trebino1993using}, radial beam shaping \cite{Liu:2002bq}, pulse shaping \cite{Rundquist:2002wx}, microscopy \cite{Gerchberg:1972uo},  to image reconstruction \cite{Fienup:1987tx}. A detailed review of the history of this problem and a number of numerical techniques that have been developed to solve this problem can be found in \cite{luke2002optical}. 

Within the context of phase retrieval, the most common technique for optimizing $I$ is the alternating projection algorithm pioneered by Gerchberg and Saxton \cite{Gerchberg:1972uo} and its variants such as the hybrid input-output algorithm discovered by Fienup \cite{Fienup:1982di}. In the original Gerchberg-Saxton algorithm the phase is recovered as follows:
\begin{enumerate}
\item Choose an initial guess $\phi_0 \in \mathcal{M}$.
\item Define $\Psi_n=\arg\left(\mathcal{F}\left[g(s)\exp\left(i \varphi_{n-1}(s)\right)\right]\right)$.
\item Define $\phi_n=\arg\left(\mathcal{F}^{-1}\left[G(\Omega)\exp\left(\left(i \Psi_n(\Omega)\right)\right)\right]\right)$.
\item Loop through items 2 and 3 until $I[\phi_n]$ is sufficiently close to zero.
\end{enumerate}

Through Plancherel's  identity it can be shown that the Gerchberg-Saxton algorithm is an error reducing algorithm in the sense that $I[\phi_{n+1}]\leq I[\phi_{n}]$ \cite{Fienup:1982di}. However, this property alone does not guarantee convergence of the algorithm. While projection algorithms converge on convex sets \cite{youla1987mathematical, bauschke2002phase}, the projections employed by the Gerchberg-Saxton algorithm are analogous to projections onto the boundary of the unit ball in $\mathbb{C}$ which is clearly not convex. This lack of convexity commonly leads to stagnation of the algorithm away from the global minimum which must be overcome by additional ad hoc means \cite{pasienski2008high, fienup1986phase, wackerman1991use}. Indeed, additional techniques such as the hybrid input-output algorithm are commonly implemented along with the Gerchberg-Saxton to force the algorithm out of local minimum \cite{fienup1986phase}. Other techniques for phase retrieval have also been developed including in the two diemnsional setting a variational principle for the transport of intensity distributions between two planes orthogonal to the optical axis \cite{rubinstein2004variational}. 

In phase shaping, however, the infimum of $I$ may be significantly bounded away from zero and therefore when applied to our problem convergence of the Gerchberg-Saxton algorithm cannot be assessed. Indeed, without any further \emph{a priori} information about the minimum value of $I$ stagnation of the algorithm is the only indication of possible convergence. Moreover, if the Gerchberg-Saxton stagnates at a large value of $I$ it may indeed be at the global minimum for the problem. 

\subsection{Organization of paper and summary of main results}
The paper is organized as follows. In section 2 we use elementary Fourier analysis to deduce some basic results concerning the functional $I$ and translate these results in terms of the practical parameters for the phase shaping problem we are interested in. Using Plancherel’s identity we obtain as a necessary condition for accurate beam shaping a normalization for the peak target intensity $E_T^2$ in terms of the design parameters:
\begin{equation}
E_T^2=\frac{ 2 \pi k \|f(\sqrt{r})\sqrt{r}\|_{L^2}^2}{\|F_T(z)\|_{L^2}^2}E_0^2.
\end{equation}
 We also prove the existence of a minimum for this problem and in doing so obtain the regularity result that the optimal phase is the argument of an analytic function. 
%\begin{equation}
%E_T^2=\frac{ 2\pi k \left\| f(r)\sqrt{r} \right\|_{L^2}^2}{\left\| F_T(z) \right\|_{L^2}^2}E_0^2.
%\end{equation}
%Second, we prove the existence of a minimum for this problem. In proving the existence of a minimum we also obtain the regularity result that the optimal $\phi$ is the argument of an analytic function. 
Finally, in using the uncertainty principle we prove \emph{ansatz free} lower bounds on the minimum value of the functional. Indeed we prove the following necessary condition for accurate beam shaping:
\begin{equation}
\beta=\frac{2k W_TW_0r_0}{4z_d^2-W_T^2}>\pi.
\end{equation}
Although we discovered this result independently,  Romero and Dickey are to the best of our knowledge the first to make the connection between the uncertainty principle and limitations in beam shaping arising from diffraction \cite{romero1996lossless}. Our work goes beyond their treatment in that we prove in a mathematically precise manner the exact range of parameters for which the minimum value is bounded away from zero. 

In section 3 we use the stationary phase method to a construct a phase $\phi(r)$ such that solutions to  Equation (\ref{intro:GoverningEquation}) with the initial data $E(r,0)=E_0f(r)\exp(i\phi(r))$ satisfy
\begin{equation*}
\lim_{k\rightarrow \infty} |E(0,z)|=E_TF_T(z).
\end{equation*}
The phase is constructed as follows:
\begin{enumerate}
\item The following initial value problem is solved:
\begin{equation*}
\begin{cases}
\displaystyle{\frac{dz_c}{d\rho}=2\pi k \frac{E_0^2 f^2(\rho)}{E_T^2 F_T^2(z_c(\rho))} \rho}\\
z_c\left(r_0-\frac{W_0}{2}\right)=z_d-\frac{W_T}{2}
\end{cases}.
\end{equation*}
\item The phase is found by integrating:
\begin{equation*}
\phi(\rho)=-k \int_{r_0-\frac{W_0}{2}}^{\rho} \frac{u}{z_c(u)}\,du.
\end{equation*}
\end{enumerate}
A similar approach was taken by Friberg in the analysis of a logarithmic axicon lens \cite{friberg1996stationary}. While Friberg showed that the logarithmic axicon lens maps the support of $f$ into $F_T$ in the short wavelength limit ($k\rightarrow \infty$), it fails to match the target intensity. Our approach is more general in that for our construction the intensity distribution is exactly matched in the short wavelength limit. We would like to point out, however, that we are not the first to discover the use of the method of stationary phase for beam shaping. Indeed, after we successfully extended Friberg's analysis in \cite{friberg1996stationary} to our setting we discovered that Romero and Dickey developed essentially the same algorithm in the two-dimensional setting in \cite{romero1996lossless} and further developed the theory in \cite{romero2000mathematical}.

In section 4 we explore how the method of stationary phase coupled with the Gerchberg-Saxton algorithm can be applied to practical problems. We consider three specific applications. We first study the problem of creating a near-uniform intensity along the optical axis down range from a source. Specifically, we take the input beam to be a Gaussian ring and the target profile to be a $n$-th order super-Gaussian. We show that the stationary phase approximation yields a good initial guess that is further improved upon by the Gerchberg-Saxton algorithm. We also highlight the role of the uncertainty principle and show that accurate phase shaping is impossible for sufficiently small values of $W_0$ or $W_T$. We also consider in this section the problem of creating beams with oscillatory intensity profiles. We show that that as a consequence of the uncertainty principle there is a cutoff in the period of these oscillations below which phase shaping cannot adequately match the target intensity. We illustrate this by applying our algorithm to periodic intensity profiles with small scale oscillations. Finally, we show in this section how the method of stationary phase coupled with linear temporal chirping can be used to form pulses with specified temporal widths at the target distance $z_d$. Specifically, by taking a separable spatio-temporal \emph{ansatz} on the initial data we show that spatial phase shaping can be employed to correct for the additional spatial spreading of the beam arising from linear temporal chirping. 

%For filamentation of the beam to occur, it is known that the self-focusing critical power of the beam is different from the estimate for a Gaussian beam. In particular, it is known from numerical experiments that the power carried in the central core of the beam must be near the critical power of the Townes profile \cite{fibich2000critical}. Using the method of stationary phase we show that at the target distance $z_d$ the power in the core of the beam can be approximated by
%\begin{equation}
%P_C(z_d)\approx 9.8\times 2^{-\frac{1}{2n}}\Gamma\left(1+\frac{1}{2n}\right)^{-1}\frac{z_d^2}{kr_0^2 W_T^{\prime}}P_0,
%\end{equation}
%where $P_0$ is the initial power of the beam. 

We conclude with a summary and discussion of our results.

%%%%%%%%%%%%%%%%%%%%%%%%%%%
%
%		Normalization
%
%%%%%%%%%%%%%%%%%%%%%%%%%%%%
\section{Mathematical results obtained through Fourier Analysis}
In this section we use the abstract variational formulation of our problem to obtain quantitative information about our beam shaping problem. Let $\mathcal{M}$ denote the space of Lebesgue measurable functions on $\mathbb{R}$ and $C_0^{\infty}(\mathbb{R}^+)$ the space of smooth compactly supported functions on $\mathbb{R}^+$. Recall  from the introduction that we are interested in minimizing the functional $I:\mathcal{M}\mapsto \mathbb{R}$ defined by
\begin{equation} \label{Form:DefFunc}
I[\varphi]=\left \|G-\left |\mathcal{F}\left[g\exp(i \varphi)\right] \right| \right\|_{L^2},
\end{equation}
with $G,g\in C_0^{\infty}(\mathbb{R}^+)$ having support bounded away from the origin. In terms of the particular beam shaping problem we are interested in the \emph{shaper phase} is given by $\phi(r)=\varphi(r^2)$, the \emph{input profile} of the electric field is given by $f(r)=E_0^{-1}g(r^2)$ and the \emph{target profile} is given by $F_T(z)=\frac{k}{2 z}E_T^{-1} G\left( \frac{k}{2z}\right)$. In particular we will assume that the supports of $g$ and $G$ are given by the intervals $S_g(r_0,W_0)$ and $S_G(z_d,W_t)$ defined by
\begin{equation} \label{Form:Sg}
S_g(r_0,W_0) = \left \{ s \in \mathbb{R} : \left(r_0-\frac{W_0}{2}\right)^2 \leq s \leq \left(r_0+\frac{W_0}{2}\right)^2 \right\}
\end{equation}
and 
\begin{equation} \label{Form:SG}
S_G(z_d,W_T) =\left\{ \Omega \in \mathbb{R}: \frac{k}{2z_d+W_T}\leq \Omega \leq \frac{k}{2z_d-W_T}\right\}
\end{equation}
respectively where $r_0,W_0,z_d,W_T>0$. 
\subsection{Consequences of Plancherel’s theorem}
Recall Plancherel’s theorem which states that $\sqrt{2\pi} \| g\|_{L^2}=\left \| \mathcal{F}[g]\right \|_{L^2}$. The following Lemma immediately follows from Plancherel’s identity and applications of the triangle inequality and reverse triangle inequality. 
\begin{lemma} \label{Form:Thm:LowerBound1} Let $I:\mathcal{M}\mapsto \mathbb{R}$ be the functional defined by Equation (\ref{Form:DefFunc}) with $g,G\in C_0^{\infty}(\mathbb{R}^+)$ having supports $S_g(r_0,W_0)$ and $S_G(z_d,W_T)$ defined by equations (\ref{Form:Sg}) and (\ref{Form:SG}) respectively. Then,
\begin{equation*}
\left| \|G\|_{L^2} -\sqrt{2\pi}\left \|g\right\|_{L^2}\right| \leq  \inf_{\phi\in \mathcal{M}}I[\phi] \leq \|G\|_{L^2} +\sqrt{2\pi} \|g\|_{L^2}.
\end{equation*}
\end{lemma}

If the lower bound in Lemma \ref{Form:Thm:LowerBound1} is zero, i.e. $\|G\|_{L^2}=\sqrt{2\pi}\|g\|_{L^2}$, then in terms of the design parameters we are interested in we have that
\begin{equation}
E_T^2=\frac{ 2\pi k \left\| f(r)\sqrt{r} \right\|_{L^2}^2}{\left\| F_T(z) \right\|_{L^2}^2}E_0^2. \label{Form:Eqn:Normalization}
\end{equation}
This provides a natural normalization for $E_T$ in terms of $E_0$ and the $L^2$ norms for the specific profiles of the input and target intensity patterns. It is important to note that in our initial formulation of the problem we did not specify any constraints on the target electric field. Throughout the rest of this paper we will assume this normalization and hence that $\|G\|_{L^2}=\sqrt{2\pi}\|g\|_{L^2}$ which will be a critical assumption in the stationary phase analysis as well.

%%%%%%%%%%%%%%%%%%%
%
%		Dual Variational Problems
%
%%%%%%%%%%%%%%%%%%%
\subsection{Equivalent variational problems and the existence of a minimum}
We can also use Plancherel's identity to obtain functionals that are defined over a larger admissible set but have the same infimum as $I$. These functionals are easier to analyze and we use them to prove the existence of a minimizer for our original problem. 

For $G,g\in C_0^{\infty}(\mathbb{R}^+)$, define the functional $\overline{I}:\mathcal{M}\times \mathcal{M} \mapsto \mathbb{R}^+$ by
\begin{equation}\label{Form:DefFunc2}
\overline{I}[\varphi,\Psi]=\left \| G\exp(i \Psi)- \mathcal{F}\left[ g\exp(i \varphi) \right] \right\|_{L^2}.
\end{equation}
Letting $\theta(\Omega)=\arg\left(\mathcal{F}\left[g\exp(i \varphi )\right](\Omega)\right)$ it follows for a fixed $\varphi\in \mathcal{M}$ that $I[\varphi, \Psi]$ is minimized by setting $\Psi(\Omega)=\phi(\Omega)$. Therefore, $\inf_{\Psi \in \mathcal{M}} \overline{I}[\varphi, \Psi]=I[\varphi]$ and consequently 
\begin{equation}
\inf_{\varphi, \Psi \in \mathcal{M}}\overline{I}\left[\varphi, \Psi \right]=\inf_{\varphi \in \mathcal{M}} I[\varphi]. 
\end{equation}

Let $B_1^{\infty}$ denote the unit ball in the $L^{\infty}$ norm. That is, 
\begin{equation*}
B_1^{\infty}=\left\{\zeta\in L^{\infty}: \esssup_{x\in \mathbb{R}} |\zeta(x)|\leq 1\right\}.
\end{equation*}
Define another functional $I^{\prime}:B_1^{\infty}\times B_1^{\infty}\mapsto \mathbb{R}$ by 
\begin{equation}\label{Form:DefFunc3}
-I^{\prime}[\zeta,w]=2 \Re\left(\left< Gw,\mathcal{F}\left[g\zeta\right] \right>\right).
\end{equation}
Since we are assuming $\|G\|_{L^2}=\sqrt{2\pi}\|g\|_{L^2}$ it follows that 
\begin{equation*}
\overline{I}\left[\varphi,\Psi\right]=2\left\|G\right\|_{L^2}+I^{\prime}[\exp(i\varphi),\exp(i\Psi)].
\end{equation*}
 Therefore, to study minimizers of $\overline{I}$ it is equivalent to study minimizers of $I^{\prime}$. Furthermore, we can use the bilinearity of $I^{\prime}$ to prove the existence of a minimizer for our functional of interest. 

\begin{theorem}[Existence of minimizer] \label{Thm:existence}  Let $I^{\prime}:B_1^{\infty}\times B_1^{\infty}\mapsto \mathbb{R}$ be the functional defined by Equation (\ref{Form:DefFunc3}) with $g,G\in C_0^{\infty}(\mathbb{R}^+)$ having supports $S_g(r_0,W_0)$ and $S_G(z_d,W_T)$ defined by equations (\ref{Form:Sg}) and (\ref{Form:SG})  respectively. There exists $\overline{\varphi},\overline{\Psi}\in \mathcal{M}$ such that for all $\zeta,w \in B_1^{\infty}$
\begin{equation*}
I^{\prime}\left[\exp(i\overline{\varphi}),\exp(i\overline{\Psi})\right]\leq I^{\prime}[\zeta,w].
\end{equation*}
\end{theorem}
\begin{proof}

 To prove the result we proceed in two steps. First we use the direct method to show that minimizers $(\overline{\zeta},\overline{\omega})$ of $I^{\prime}$ over  $B_1^{\infty}\times B_1^{\infty}$ exist. Second we show that these minimizers satisfy $|\overline{\zeta}(s)|=1$, $|\overline{\omega}(\Omega)|=1$ $a.e.$ and hence $\left(\overline{\varphi},\overline{\Psi}\right)\in \mathcal{M}\times \mathcal{M}$ defined by $\overline{\varphi}=\arg(\overline{\zeta})$, $\overline{\Psi}=\arg(\overline{\omega})$ are the desired minimum. 

Let $(\zeta_n,\omega_n)$ be a minimizing sequence for $I^{\prime}$ over $B_1^{\infty}\times B_1^{\infty}$. That is, 
\begin{equation*}
\lim_{n\rightarrow \infty}I^{\prime}[\zeta_n,\omega_n]=\inf_{\zeta,\omega\in B_1^{\infty}}I^{\prime}[\zeta,\omega].
\end{equation*}
By the Banach-Alaoglu theorem $B_1^{\infty}$ is weak-$*$ compact and hence $B_1^{\infty}\times B_1^{\infty}$ is also compact in the product topology. Therefore, reindexing a subsequence if necessary, we can assume without loss of generality that there exists $(\overline{\zeta},\overline{\omega})\in B_1^{\infty}\times B_1^{\infty} $  such that $\zeta_{n}\overset{*}{\rightharpoonup} \overline{\zeta}$  and $\omega_{m}\overset{*}{\rightharpoonup} \overline{\omega}$ \cite{evans1990weak}.

Let $h_{n}=\mathcal{F}[g\zeta_n]$ and $h=\mathcal{F}[g\overline{\zeta}]$. From weak-$*$ convergence it follows that for fixed $\Omega$ that
\begin{align*}
\lim_{n\rightarrow \infty}h_n(\Omega)=\lim_{n\rightarrow \infty}\int_{-\infty}^{\infty}g(s)\zeta_n(s)e^{i \Omega s}\,ds=\int_{-\infty}^{\infty} g(s)\overline{\zeta}(s)e^{i \Omega s}\,ds= h(\Omega).
\end{align*}
Furthermore, for all $\Omega$ it follows from H{o}lder’s inequality that 
\begin{equation*}
|G(\Omega)h_n(\Omega)-G(\Omega)h(\Omega)|\leq 2\|G\|_{L^{\infty}} \|g\|_{L^1}. 
\end{equation*}
Therefore, by Lebesgue’s Dominated Convergence Theorem it follows that
\begin{equation*}
Gh_n \overset{L^1}{\rightarrow} Gh.
\end{equation*}
Finally, from weak-$*$ convergence of $w_n$ and the strong convergence of $Gh_n$ it follows that
\begin{equation*}
I^{\prime}[\overline{\zeta},\overline{w}]=\lim_{n\rightarrow \infty}I^{\prime}[\zeta_n,\omega_n]=\inf_{\zeta,\omega\in B_1^{\infty}}I^{\prime}[\zeta,\omega].
\end{equation*}
That is, $(\overline{\zeta},\overline{w})$ minimizes $I^{\prime}$ over $B_1^{\infty}\times B_1^{\infty}$. 

Since $g\in C_0^{\infty}(\mathbb{R}^+)$ we deduce from the Paley-Wiener Theorem that $h$ is analytic and hence has only isolated zeros \cite{paley1934fourier}. 
Let $\theta(\Omega)=\arg(h(\Omega))$ and $\psi(\Omega)=\arg(\overline{w}(\Omega))$. Therefore, 
\begin{equation*}
I^{\prime}[\overline{\zeta},\overline{w}]=-2\int_{\mathbb{R}} \left| h(\Omega) \right| \left| w(\Omega) \right| G(\Omega)\cos\left( \theta(\Omega)-\Psi(\Omega)\right)\,d\Omega
\end{equation*}
is minimized if and only if $\left| \overline{w}(\Omega)\right|=1$ a.e. and there exists an integer valued function $m(\Omega)$ such that $\theta(\Omega)=\Psi(\Omega)+2m(\Omega)\pi$. From Parseval’s equality a similar argument proves that $|\zeta(s)|=1$ a.e.. 
\end{proof}

\begin{corollary} Let $I:\mathcal{M} \mapsto \mathbb{R}^+$ be the functional defined by Equation (\ref{Form:DefFunc}) with $g,G\in C_0^{\infty}(\mathbb{R}^+)$ having supports $S_g(r_0,W_0)$ and $S_G(z_d,W_T)$ defined by equations (\ref{Form:Sg}) and (\ref{Form:SG}) respectively. 
There exists $\overline{\varphi}\in \mathcal{M}$ such that for all $\varphi\in \mathcal{M}$ 
\begin{equation*}
I[\overline{\varphi}]< I[\varphi].
\end{equation*}
\end{corollary}
\begin{proof}
Let $\overline{I}:\mathcal{M}\times \mathcal{M}$ be the functional defined by Equation (\ref{Form:DefFunc2}) and let $\left(\overline{\varphi},\overline{\Psi}\right)\in \mathcal{M}\times \mathcal{M}$ be the global minimizer for $\overline{I}$. If we let 
\begin{equation*}
\theta(\Omega)=\arg\left(\mathcal{F}\left[g\exp(i \varphi)\right]\right)(\Omega)
\end{equation*}
 then without loss of generality we can assume $\Psi=\theta$ and hence
\begin{equation*}
I[\overline{\varphi}]=\overline{I}[\overline{\varphi},\overline{\Psi}]\leq \inf_{\varphi\in \mathcal{M}}I[\varphi].
\end{equation*}
\end{proof}

\begin{corollary} If $I:\mathcal{M} \mapsto \mathbb{R}^+$ is the functional defined by Equation (\ref{Form:DefFunc}) with $g,G\in C_0^{\infty}(\mathbb{R}^+)$ having supports $S_g(r_0,W_0)$ and $S_G(z_d,W_T)$ defined by equations (\ref{Form:Sg}) and (\ref{Form:SG}) respectively then
\label{cor:lowerbound1}
\begin{equation*}
\inf_{\varphi \in \mathcal{M}}I[\varphi]>0.
\end{equation*}
\end{corollary}
\begin{proof}
Let $\overline{\varphi}\in \mathcal{M}$ be the global minimizer of $I$ over $\mathcal{M}$. Since $g\in C_0^{\infty}(\mathbb{R}^+)$ the Paley-Weiner theorem implies that  $\mathcal{F}[g\exp\left(i \overline{\varphi} \right)]$ is analytic. Therefore, $\mathcal{F}[g\exp\left(i \overline{\varphi} \right)]$ cannot be compactly supported on $\mathbb{R}$ and since $G\in C_0^{\infty}(\mathbb{R}^+)$ is compactly supported on $S_G(z_d,W_T)$ it follows that 
\begin{equation*}
I^2[\overline{\phi}]\geq \int_{\mathbb{R}\backslash S_G(z_d,W_T)}\left| \mathcal{F}[g\exp\left(i \overline{\varphi}\right)]\right|^2(\Omega)\,d\Omega>0.
\end{equation*}
\end{proof}

\begin{corollary}[Regularity]\label{Cor:Regularity}
Let $I:\mathcal{M} \mapsto \mathbb{R}^+$ be the functional defined by Equation (\ref{Form:DefFunc}) with $g,G\in C_0^{\infty}(\mathbb{R}^+)$ having supports $S_g(r_0,W_0)$ and $S_G(z_d,W_T)$ respectively. If $\overline{\varphi}\in \mathcal{M}$ minimizes $I$ then $\exp\left(i\overline{\varphi}(s)\right)$ is analytic.
\end{corollary}
\begin{proof}
Since $\overline{\varphi}(s)$ minimizes $I$ over $\mathcal{M}$ it follows from Theorem \ref{Thm:existence} that there exists $\overline{\Psi}\in \mathcal{M}$ such that $(\overline{\varphi},\overline{\Psi})$ minimizes $I^{\prime}$. Consequently, from the same argument used in the proof of Theorem \ref{Thm:existence} and Parseval’s theorem it follows that 
\begin{equation*}
\overline{\varphi}(s)=\arg\left( \mathcal{F}^{-1}\left[G \exp\left(i \overline{\Psi}\right)\right](s)\right).
\end{equation*}
Since $G\in C_0^{\infty}(\mathbb{R}^+)$ it follows from the Paley-Wiener Theorem that $\exp(i \overline{\phi}(s))$ is analytic. 
\end{proof}

It follows from the previous corollary that the minimizers of $I$ are well behaved in the sense that they are smooth except  possibly for jump discontinuities of $2\pi$.  

%\begin{equation}
%I^{\prime}[\Psi]=\left\| \left|\mathcal{F}^{-1}\left[G(\Omega) \exp\left(i \Psi(\Omega)\right)\right](s)\right|-g(s) \right\|_{L^2}.
%\end{equation}

%Define another functional $I^{\prime}:\mathcal{M}\mapsto \mathbb{R}^{+}$ by 
%\begin{equation}
%I^{\prime}[\Psi]=\left\| \left|\mathcal{F}^{-1}\left[G(\Omega) \exp\left(i \Psi(\Omega)\right)\right](s)\right|-g(s) \right\|_{L^2}.
%\end{equation}
%This functional is essentially the same as $I$ except that it is using the inverse Fourier transform to map the target function back to the initial data. 
%
%\begin{lemma} \label{Form:Dual2} 
%\begin{equation*}
%\inf_{\phi \in \mathcal{M}}I[\phi] = \sqrt{2\pi} \inf_{\Psi \in \mathcal{M}} I^{\prime}[\Psi]\end{equation*}
%\end{lemma}
%\begin{proof}
%By  lemma \ref{Form:Dual1}, Plancherel's theorem, and a similar argument used to prove lemma \ref{Form:Dual1} it follows that
%\begin{align*}
%\inf_ {\phi\in \mathcal{M}} I[\phi] &=  \inf_{\phi, \Psi \in \mathcal{M}} \overline{I}[\phi,\Psi]\\
%&= \sqrt{2\pi}\inf_{\phi, \Psi \in \mathcal{M}}  \left \|\mathcal{F}^{-1}\left[G(\Omega)\exp\left(i \Psi(\Omega)\right)\right](s)-g(s) e^{i \phi(\sqrt{s})} \,\right\|_{L^2}\\
%&= \sqrt{2\pi}\inf_{\Psi \in \mathcal{M}} \left \| \,\left|\mathcal{F}^{-1}\left[G(\Omega)\exp\left(i \Psi(\Omega)\right)\right](s)\right|-g(s)\,\right\|_{L^2}\\
%&= \sqrt{2\pi}\inf_{\Psi \in \mathcal{M}}I^{\prime}[\Psi].
%\end{align*}
%\end{proof}

%%%%%%%%%%%%%%%%%%%%%
%
%	Uncertainty Principle
%
%%%%%%%%%%%%%%%%%%%%%

\subsection{Consequences of the Uncertainity Principle}
The uncertainty principle quantifies the impossibility of localizing both a function and its Fourier transform \cite{folland1997uncertainty}. Therefore, if the length of the support of $G$ and $g$ are both sufficiently small then the minimum value of $I$ will necessarily be significantly bounded away from zero. In particular, while Corollary \ref{cor:lowerbound1} quantifies that the minimum value of $I$ is not zero the uncertainty principle can give us quantitative information about how large the minimum value can be. This is made precise by the following theorem.

\begin{theorem} \label{thm:Unc1} Let $I:\mathcal{M}\mapsto \mathbb{R}^+$ be the functional defined by Equation (\ref{Form:DefFunc}) with $g,G\in C_0^{\infty}(\mathbb{R}^+)$ having supports $S_g(r_0,W_0)$ and $S_G(z_d,W_T)$ respectively. If $\|G\|_{L^2}=\sqrt{2\pi} \|g\|_{L^2}$, we have the lower bound
\begin{equation*}
\inf_{\varphi \in \mathcal{M}}I[\varphi]\geq \|G\|_{L^2}\left(1-\frac{1}{\sqrt{\pi}}\sqrt{\frac{2kW_Tr_0W_0}{4z_d^2-W_T^2}}\right).
\end{equation*}
\end{theorem}
\begin{proof}
By the reverse triangle inequality it follows that 
\begin{align*}
I[\varphi]&= \left\|G-\mathcal{F}\left[g\exp\left( i \varphi\right)\right]\right\|_{L^2}\\
&\geq  \left(\int_{S_G(z_d,W_T)} \left(G(\Omega)-\left|\mathcal{F}\left[g\exp\left( i \varphi \right)\right](\Omega)\right| \right)^2\,d\Omega \right)^{\frac{1}{2}}\\
&\geq  \left| \left(\int_{S_G(z_d,W_T)} G^2(\Omega)\,d\Omega \right)^{\frac{1}{2}} - \left(\int_{S_G(z_d,W_T)} \left|\mathcal{F}\left[g\exp\left( i \varphi \right)\right](\Omega)\right|^2 \,d\Omega \right)^{\frac{1}{2}} \right|\\
&= \left| \|G\|_{L^2} - \left(\int_{S_G(z_d,W_T)} \left|\mathcal{F}\left[g \exp\left( i \varphi \right)\right](\Omega)\right|^2 \,d\Omega \right)^{\frac{1}{2}} \right|.\\
\end{align*}
Furthermore, by direct calculation, an application of the Cauchy-Schwarz inequality and Lemma \ref{Form:Eqn:Normalization} we have that
\begin{align*}
\left(\int_{S_G(z_d,W_T)} \left|\mathcal{F}\left[g \exp\left( i \varphi \right)\right](\Omega)\right|^2 \,d\Omega \right)^{\frac{1}{2}} &\leq \sqrt{ \frac{2k W_T}{4z_d^2-W_T^2}}\left \| \mathcal{F}\left[g \exp\left( i \varphi \right)\right] \right\|_{L^{\infty}}\\
& \leq \sqrt{ \frac{2k W_T}{4z_d^2-W_T^2}}\|g(s)\|_{L^1}\\
&= \sqrt{ \frac{2k W_T}{4z_d^2-W_T^2}}\int_{S_g(r_0,W_0)}\left|g(s)\right|\,ds\\
&\leq \sqrt{ \frac{4k W_Tr_0W_0}{4z_d^2-W_T^2}}\|g \|_{L^2}\\
&= \sqrt{ \frac{4k W_Tr_0W_0}{4z_d^2-W_T^2}}\frac{1}{\sqrt{2\pi}}\|G \|_{L^2}.
\end{align*}
\end{proof}

Using the same notation as in \cite{romero1996lossless} we define the dimensionless parameter $\beta$ by
\begin{equation*}
\beta=\frac{2kr_0W_TW_0}{4z_d^2-W_T^2}.
\end{equation*}
It then follows from Theorem~\ref{thm:Unc1}, that $\displaystyle{I[\varphi]\geq \|G\|_{L^2}\left(1-\sqrt{\frac{\beta}{\pi}}\right)}$, for any choice of the phase function $\varphi$. In particular, if $\beta < \pi$, there is no way of using phase shaping to match the target profile. 

We emphasize that the result is only a lower bound for the functional $I[\varphi]$, and in particular it does not follow that if the parameters $k,r_0,W_0,W_T$ and $z_d$ are such that $\beta > \pi$, then the infimum of $I$ is zero. Also, this result is somewhat crude in that it is independent of the exact forms of $G$ and $g$. Below is a refinement of Theorem \ref{thm:Unc1} that takes into account the $L^1$ norms of $g$ and $G$.

\begin{theorem}\label{thm:Unc2} Let $I:\mathcal{M}\mapsto \mathbb{R}^+$ be the functional defined by Equation (\ref{Form:DefFunc}) with $g,G\in C_0^{\infty}(\mathbb{R}^+)$ having supports $S_g(r_0,W_0)$ and $S_G(z_d,W_T)$ defined by equations (\ref{Form:Sg}) and (\ref{Form:SG}) respectively.
\begin{enumerate}
\item If  $\|G\|_{L^2}=\sqrt{2\pi} \|g\|_{L^2}$ 
%%and 
%\begin{equation*}
%\displaystyle{\frac{kW_T}{4z_d^2-W_T^2}\|g\|_{L^1}^2\leq \pi \|g\|_{L^2}^2}
%\end{equation*}
then
\begin{equation*}
\inf_{\varphi \in \mathcal{M}}I[\varphi]\geq \sqrt{2\pi}\|g\|_{L^2} \left(1- \sqrt{\frac{2kW_T}{4z_d^2-W_T^2}}\frac{\|g\|_{L^1}}{\left\|g\right\|_{L^2}}\right).
\end{equation*} 
\item If If $\|G\|_{L^2}=\sqrt{2\pi} \|g\|_{L^2}$ 
%and 
%\begin{equation*}
%r_0W_0\|G\|_{L^1}^2\leq \pi \|G\|_{L^2}^2
%\end{equation*}
then
\begin{equation*}
\inf_{\varphi \in \mathcal{M}}I[\varphi]\geq \left\|G\right\|_{L^2}\left( 1 - \sqrt{\frac{r_0W_0}{\pi}}\frac{\|G\|_{L^1} }{\left\| G \right\|_{L^2}}\right).
\end{equation*}
\end{enumerate}
\end{theorem}
\begin{proof}
Item 1 follows from the proof of Theorem \ref{thm:Unc1} and Plancherel’s identity. To prove item 2 we note that from Plancherel’s identity and the triangle inequality that
\begin{equation*}
\overline{I}[\varphi,\Psi]\geq  \sqrt{2\pi}  \left|\left(\int_{S_g(r_0,W_0)} \left| \mathcal{F}^{-1}[G\exp\left(i \Psi \right)](s)\right|^2\,ds\right)^{\frac{1}{2}}-\|g\|_{L^2}\right|.
\end{equation*}
Following the same arguments as in the proof of Theorem \ref{thm:Unc1} it follows that
\begin{equation*}
\left(\int_{S_g(r_0,W_0)} \left| \mathcal{F}^{-1}[G \exp\left(i \Psi \right)](s)\right|^2\,ds\right)^{\frac{1}{2}}\leq \frac{1}{2\pi} \sqrt{2r_0W_0}\|G \|_{L^1}.
\end{equation*}
The result follows from applying Plancherel’s identity to $\|g\|_{L^2}$.
\end{proof}

Theorems \ref{thm:Unc1} and \ref{thm:Unc2} are estimates for the global error in the beam shaping problem in terms of the design parameters. The same techniques can be used to obtain local information as well. Specifically, the uncertainty principle implies that if $G$ has sufficiently small scale features within its support then the infimum of $I$ will be bounded away from zero. To make this statement precise we look at the local error for an interval of width $W_T^{\prime}$ centered at $z_d^{\prime}$ on the optical axis. Specifically,  for all $W_T^{\prime}, z_d^{\prime} \in \mathbb{R}^+$ satisfying 
\begin{equation}
2z_d^{\prime}-W_T^{\prime}>2z_d-W_T \text{ and } 2z_d^{\prime}-W_T^{\prime}< 2z_d-W_T
\end{equation}
define $I^{W_T^{\prime}}_{z_d^{\prime}}: \mathcal{M}\mapsto \mathbb{R}^{+}$ by
\begin{align}\label{Form:Func4}
I^{W_T^{\prime}}_{z_d^{\prime}}[\varphi]&=\left \|G-\left |\mathcal{F}\left[g\exp(i \varphi)\right] \right| \right\|_{L^2(S_G(z_d^{\prime},W_T^{\prime}))}\\
&=\left(\int_{S_G(z_d^{\prime},W_T^{\prime})}\left(G(\Omega)-\left |\mathcal{F}\left[g\exp(i \varphi)\right](\Omega) \right|\right)^2\,d\Omega\right)^{\frac{1}{2}},
\end{align}
which is a local measure of the $L^2$ error on the interval $|z-z_d^{\prime}|< W_T^{\prime}$ along the optical axis. The following corollary follows from the same arguments used to prove Theorem \ref{thm:Unc1}.
\begin{corollary} \label{cor:unc1} Let $I_{z_d^{\prime}}^{W_T^{\prime}}:\mathcal{M}\mapsto \mathbb{R}^+$ be the functional defined by Equation (\ref{Form:Func4}) with $g,G\in C_0^{\infty}(\mathbb{R}^+)$ having supports $S_g(r_0,W_0)$ and $S_G(z_d,W_T)$ defined by equations (\ref{Form:Sg}) and (\ref{Form:SG}) respectively. Suppose $z_d^{\prime}, W_T^{\prime}$ satisfy $2z_d^{\prime}-W_T^{\prime}>2z_d-W_T$ and $2z_d^{\prime}-W_T^{\prime}< 2z_d-W_T$. If $\|G\|_{L^2}=\sqrt{2\pi} \|g\|_{L^2}$ and 
\begin{equation*}
\frac{2kW_T^{\prime}W_0r_0}{4\left(z_d^{\prime}\right)^2-\left(W_T^{\prime}\right)^2}\frac{\|G\|_{L^2}^2}{\|G(\Omega)\|_{L^2(S_G(z_d^{\prime},W_T^{\prime}))}^2}< \pi,
\end{equation*}
then 
\begin{equation*}
\inf_{\varphi \in \mathcal{M}}I_{z_d^{\prime}}^{W_T^{\prime}}[\varphi]\geq \|G\|_{L^2(S_G(z_d^{\prime},W_T^{\prime}))}-\frac{1}{\sqrt{\pi}}\sqrt{\frac{2kW_T^{\prime}r_0z_d^{\prime}}{4\left(z_d^{\prime}\right)^2-\left(W_T^{\prime}\right)^2}}\|G\|_{L^2}.
\end{equation*}
\end{corollary}

\begin{corollary} \label{cor:unc2} Let $I_{z_d^{\prime}}^{W_T^{\prime}}:\mathcal{M}\mapsto \mathbb{R}^+$ be the functional defined by Equation (\ref{Form:Func4}) with $g,G\in C_0^{\infty}(\mathbb{R}^+)$ having supports $S_g(r_0,W_0)$ and $S_G(z_d,W_T)$ defined by equations (\ref{Form:Sg}) and (\ref{Form:SG}) respectively. Suppose $z_d^{\prime}, W_T^{\prime}$ satisfy $2z_d^{\prime}-W_T^{\prime}>2z_d-W_T$ and $2z_d^{\prime}-W_T^{\prime}< 2z_d-W_T$. If $\|G\|_{L^2}=\sqrt{2\pi} \|g\|_{L^2}$ and 
\begin{equation*}
\frac{2kW_T^{\prime}}{4\left(z_d^{\prime}\right)^2-\left(W_T^{\prime}\right)^2}\|g\|_{L^1}^2\leq \|G\|_{L^2(S_G(z_d^{\prime},W_T^{\prime}))}^2,
\end{equation*}
then 
\begin{equation*}
\inf_{\varphi \in \mathcal{M}}I_{z_d^{\prime}}^{W_T^{\prime}}[\varphi]\geq \|G\|_{L^2(S_G(z_d^{\prime},W_T^{\prime}))}-\sqrt{\frac{2kW_T^{\prime}}{4\left(z_d^{\prime}\right)^2-\left(W_T^{\prime}\right)^2}}\|g\|_{L^1}.
\end{equation*}
\end{corollary}

\begin{remark} \label{form:Remark} The local error $I^{W_T^{\prime}}_{z_d^{\prime}}$ also provides a lower bound for the global error $I$. Indeed, if we let 
\begin{equation*}
\mathcal{A}=\left \{ (z_d^{\prime}, W_T^{\prime})\in \mathbb{R}^+\times \mathbb{R}^+: 2z_d^{\prime}-W_T^{\prime}\geq 2z_d-W_T \text{ and } 2z_d^{\prime}+W_T^{\prime}\leq 2z_d+W_T \right\}
\end{equation*}
then a sharper estimate for the lower bound on the global error can be found by considering:
\begin{align}
\inf_{\varphi \in \mathcal{M}}I[\varphi]&\geq \sup_{ (z_d^{\prime}, W_T^{\prime})\in \mathcal{A} } \inf_{\varphi \in \mathcal{M}}I^{W_T^{\prime}}_{z_d^{\prime}}[\varphi]\\
& \geq  \sup_{(z_d^{\prime},W_T^{\prime})\in \mathcal{A}} \left(\|G\|_{L^2(S_G(z_d^{\prime},W_T^{\prime}))}-\frac{1}{\sqrt{\pi}}\sqrt{\frac{2kW_Tr_0z_d^{\prime}}{4\left(z_d^{\prime}\right)^2-\left(W_T^{\prime}\right)^2}}\|G\|_{L^2}\right). \nonumber
\end{align}
\end{remark}

\section{Stationary Phase Construction}
Recall that the method of stationary phase can be used to obtain asymptotic expansions for integrals of the form
\begin{equation*}
F(\lambda)=\int_{-\infty}^{\infty} e^{i \lambda H(x)}h(x)\,dx,
\end{equation*}
where $\lambda>0$, $x$ is real, $H$ is a smooth real valued function and $h$ is a smooth but not necessarily analytic complex valued function. If $H$ has one critical point at $x_c$, i.e. $\left. \frac{dH}{dx}\right|_{x_c}=0$, and $\left. \frac{d^2H}{dx^2}\right|_{x_c}\neq 0$ then 
\begin{equation*}
F(\lambda)=\exp\left(i \lambda H(x_c)+i \left(\left.\frac{d^2 H}{dx^2}\right|_{x_c}\right)\right)\sqrt{\frac{2\pi}{\lambda \left.\frac{d^2 H}{dx^2}\right|_{x_c}}}h(x_c)+\mathcal{O}\left(\frac{1}{\lambda}\right)
\end{equation*}
as $\lambda\rightarrow \infty$ \cite{miller2006applied}. In this section we show how the method of stationary phase can be used to approximate the on-axis profile of the electric field in the limit of short wavelengths and we show how this approximation can be used to construct initial guesses for the optimal phase.

\subsection{Short wavelength limit}
First, we must define precisely what we mean by short wavelengths in terms
of the input parameters and the length scales in the parameter. Using the
geometry of the problem as in Figure~\ref{fig:Besselsubfig1},  we define
the dimensionless quantities
\begin{equation}
\begin{array}{cc}
\displaystyle{\varrho=\frac{\rho-r_0}{W_0}}, &
\displaystyle{\zeta=\frac{z-z_d}{W_T}}.\\
\end{array}
\end{equation}
Equation (\ref{Intro:OpticalAxis}) gives the on-axis electric field:
% Split this equation into 2 lines
\begin{align}
E(0,z_d+W_T\zeta) =&-\frac{ik E_0 W_0}{z_d + W_T
\zeta}\int_{-\frac{1}{2}}^{\frac{1}{2}}f(r_0 + W_0 \varrho) \nonumber \\
& \times \exp\left(i \left( \phi+k \frac{(r_0+ W_0\varrho)^2}{2 (z_d + W_T
\zeta)}\right)\right)(r_0+ W_0 \varrho)\,d\varrho,
\end{align}
where we used the fact that the input profile $E(r,0)$ is only supported
on $[r_0-W_0/2, r_0+W_0/2]$.

Assuming $W_0/r_0 \ll 1, W_T/z_d \ll 1$, we expand the phase in powers of
$W_0/r_0$ and $W_T/z_d$ to obtain %replaced 1 1 by 1
\begin{align*}
\phi+k \frac{(r_0+ W_0\varrho)^2}{2 (z_d + W_T \zeta)} & = \phi + \frac{k
r_0^2}{2 z_d} - \frac{k r_0^2 W_T \zeta}{2 z_d^2} +  \frac{k r_0 W_0
\varrho}{z_d} - \frac{k r_0 W_0 W_T \zeta \varrho}{ z_d^2} + \cdots .
\end{align*}
Note that the first three terms in the expansion of $\frac{k \rho^2}{2 z}$
do not depend  on both $\varrho$ and $\zeta$ jointly. %changed wording
Consequently, they can be absorbed as phase factors in $E_0(r)$ or
$E(0,z)$ as below, and they {\em do not affect the intensity} of the input
$E_0(r)$ or the on-axis profile $E(0,z)$. Dropping the higher order terms
in $W_0/r_0$ and $W_T/z_d$, we get
\begin{align*}
z E(0,z) e^{i \frac{k r_0^2 W_T \zeta}{2 z_d^2}} \approx & -ik E_0 W_0
e^{i \frac{k r_0^2}{2 z_d}} \int_{-\frac{1}{2}}^{\frac{1}{2}}\left[ f(r_0
+ W_0 \varrho) e^{i \frac{k r_0 W_0 \varrho}{z_d}}\right] \\
& \times  e^{i \left(\phi -  \frac{k r_0 W_0 W_T \zeta \varrho}{ z_d^2}
\right)} (r_0+ W_0 \varrho)\,d\varrho.
\end{align*}
 % split equation into 2 lines, found one $\bar{\rho}$ which I replaced
with $\varrho$.
 In order to use phase variations for optical beam shaping, it is clear
that the optical path differences over the relevant region
$(\varrho,\zeta)$ should be many times the wavelength, so that the beam
can interfere with itself. This allows for the possibility of controlling
the intensity of the beam as a function of position. Since the variations
in $\zeta$ and $\varrho$ are $O(1)$ by construction, the preceding
calculation determines the correct notion of the short wavelength limit
in our problem, namely
\begin{equation*}
2 \kappa=\frac{k r_0 W_0 W_T}{ z_d^2}  \gg 2 \pi.
\end{equation*}
We include an additional factor of  2 in defining $\kappa$ so that for
$W_T/z_d \ll 1$, the dimensionless parameter $\beta$, whose definition was
motivated by Theorem~\ref{thm:Unc1} {\em is identical} to the
non-dimensional wave number $\kappa$ arising from the analysis of beam
shaping using phase variations. This argument therefore gives an
alternative physical interpretation of the parameter $\beta$ -- It is the
phase variation within the beam arising from the different optical path
lengths for different parts of the beam in our geometry
(Fig.~\ref{fig:Besselsubfig1}).

\subsection{Stationary phase approximation}
We now assume $W_0/r_0\ll 1$, $W_T/z_d\ll 1$, $\beta \gg \pi$ and formally apply the method of stationary phase to approximate the on-axis electric field. Define dimensionless quantities by
\begin{equation}
\begin{array}{cccc}
\displaystyle{\bar{\rho}=\frac{\rho}{r_0}}, & \displaystyle{\bar{z}=\frac{z}{z_d}}, & \displaystyle{\bar{k}=\frac{r_0^2k}{z_d}}, & \displaystyle{\bar{\phi}=\frac{z_d}{k r_0^2} \phi},
\end{array}
\end{equation}
then by Equation (\ref{Intro:OpticalAxis}) it follows that the on-axis profile of solutions to the paraxial wave equation  is given by
\begin{equation}
E(0,\bar{z})=-\frac{\bar{k}E_0}{\bar{z}}\int_0^{\infty}f(\bar{\rho})\exp\left(i \bar{k} \left( \bar{\phi}+\frac{\bar{\rho}^2}{2 \bar{z}}\right)\right)\bar{\rho}\,d\bar{\rho}.
\end{equation}

If $\overline{\phi}$ is a given smooth function, then for a fixed value of $\bar{z}$ stationary points of the phase can be found by solving the equation
\begin{equation}\label{Stat:StatPointEquation}
\frac{d \bar{\phi}}{d \bar{\rho}}+\frac{\bar{\rho}}{\bar{z}}=0
\end{equation}
for $\bar{\rho}$. If $\overline{\phi}$ is a monotone decreasing function in $\bar{\rho}$ and $\bar{\phi}$ is either strictly convex or concave then there exists one solution to this equation for each value of $\overline{z}$  and we can implicitly define a function $\bar{\rho}_c(\bar{z})$ by
\begin{equation}
\bar{\rho}_c(\bar{z})=-\bar{z}\left.\frac{d \bar{\phi}}{d \bar{\rho}}\right|_{\rho_c(\bar{z})}.
\end{equation}
Therefore, if we let $\Psi\left(\bar{\rho},\bar{z}\right)= \bar{\phi} + \bar{\rho}^2/(2\bar{z})$ then by the method of stationary phase it follows that if  
\begin{equation}\label{Stat:ConditionSecDerivative}
\left.\frac{d^2 \Psi}{d \bar{\rho}^2}\right|_{\bar{\rho}_c(\bar{z})}\neq 0
\end{equation}
 then 
\begin{equation} \label{Stat:Approximation1}
\frac{\bar{z}}{E_0}\left| E(0,\bar{z}) \right| = \sqrt{\frac{ 2\pi \bar{k}}{ \left| \frac{ \partial^2 \Psi}{\partial \bar{\rho}^2}\right|_{\bar{\rho}_c(\bar{z})}}}f\left(\bar{\rho}_c(\bar{z})\right)\bar{\rho}_c(\bar{z})+\mathcal{O}\left(1\right)
\end{equation}
as $\bar{k}\rightarrow \infty$.

Equation (\ref{Stat:Approximation1}) gives quantitive information about how rays of light coming from the ring-beam are mapped to the optical axis in the short wavelength limit. In particular, in the short wavelength limit the function $\bar{\rho}_c(\bar{z})$ can be thought of as a mapping between points on the optical-axis and the input plane. However, we are interested in the inverse relationship. Since $\bar{\phi}$ is either strictly convex or concave it follows by differentiating Equation (\ref{Stat:StatPointEquation}) with respect to $\bar{\rho}$ that Equation (\ref{Stat:ConditionSecDerivative}) is equivalent to the condition that $\bar{\rho}_c(\bar{z})$ is either monotone increasing or decreasing. Therefore, it follows from the inverse function theorem that $\bar{\rho}_c(\bar{z})$ is invertible with inverse $\bar{z}_{c}(\bar{\rho})$ and by Equation (\ref{Stat:StatPointEquation}) we have that 
\begin{equation}\label{Stat:PhaseDifferentialEquation}
\frac{d\bar{\phi}}{d \bar{\rho}}=-\frac{\bar{\rho}}{\bar{z}_c(\bar{\rho})}.
\end{equation}
Consequently,
\begin{equation}
\frac{d^2 \Psi}{d \bar{\rho}^2}=\frac{\bar{\rho}}{\bar{z}_c^2({\bar{\rho})}}\frac{d \bar{z}_c}{d\bar{\rho}}.
\end{equation}
Therefore, in terms of $\bar{z}_c(\bar{\rho})$ and Equation (\ref{Stat:Approximation1}) the on-axis electric field can be approximated by
\begin{equation}\label{Stat:Approximation2}
|E\left(0,\bar{z}\right)|\approx E_0 \sqrt{2\pi \bar{k}} f(\bar{\rho}_c(z))\bar{\rho}_c^{\frac{1}{2}}(\bar{z})\left(\left.\frac{d \bar{z}_{c}}{d\bar{\rho}}\right|_{\bar{\rho}_c(\bar{z})}\right)^{-\frac{1}{2}}.
\end{equation}

\subsection{Algorithm for stationary phase approximation}
Equation (\ref{Stat:Approximation2}) can be used to construct an accurate approximation to an optimal $\bar{\phi}$. Formally, if we set $|E(0,z)|=E_T F_T(z)$ then the following initial value problem must be satisfied:
\begin{equation}\label{Stat:zDifferentialEquation}
\begin{cases}
\displaystyle{\frac{d \bar{z}_c}{d \bar{\rho}}=2\pi \bar{k} \frac{E_0^2f^2(\bar{\rho})}{E_T^2 F_T^2(\bar{z}_c)}\bar{\rho}}\\
\bar{z}_c\left(1-\frac{W_0}{2r_0}\right)=1-\frac{W_T}{2z_d}
\end{cases}.
\end{equation}
The phase corresponding to $\bar{z}_c$ can then be found by integration:
\begin{equation}
\bar{\phi}(\bar{\rho})=-\int_{1-\frac{W_0}{2r_0}}^{\bar{\rho}}\frac{u}{\bar{z}_c(u)}\,du. \label{Stat:phaseDifferentialEquation}
\end{equation}
equations (\ref{Stat:zDifferentialEquation}) and (\ref{Stat:phaseDifferentialEquation}) are a simple system of differential equations whose solution we will use as an approximation to the near optimal phase.

Solutions to (\ref{Stat:zDifferentialEquation}) must also satisfy the additional physical constraint that $\bar{z}_c\left(1+\frac{W_0}{2r_0}\right)=1+\frac{W_T}{2z_d}$, that is, the stationary phase approximation maps the support of $f$ to the support of $F_T$. The following lemma shows that this condition is met if the normalization (\ref{Form:Eqn:Normalization}) is satisfied and again highlights the important role this normalization takes.
\begin{lemma}
If $\bar{z}_c$ is a solution to the initial value problem (\ref{Stat:zDifferentialEquation}) and $E_T$ and $E_0$ satisfy Equation (\ref{Form:Eqn:Normalization}) then $\bar{z}_c\left(1+\frac{W_0}{2r_0}\right)=1+\frac{W_T}{2z_d}$.
\end{lemma}
\begin{proof} Suppose $\bar{z}_c$ is a solution to the initial value problem (\ref{Stat:zDifferentialEquation}) and $\bar{z}_c\left(1+\frac{W_0}{2r_0}\right)=z^*$. From the normalization (\ref{Form:Eqn:Normalization}) and the fact that $z_c$ solves (\ref{Stat:zDifferentialEquation}) it follows that
\begin{align*}
\int_{1-\frac{W_T}{2z_d}}^{\bar{z}_c(\bar{\rho})}F_T^2(\bar{z})\,d\bar{z}&=2\pi \bar{k} \frac{E_0^2}{E_T^2}\int_{1-\frac{W_0}{2r_0}}^{\bar{\rho}}f^2(u)u\,du\\
\Rightarrow \frac{\int_{1-\frac{W_T}{2z_d}}^{\bar{z}_c(\bar{\rho})}F_T^2(z)\,dz}{\|F_T(z)\|_{L^2}^2}&= \frac{\int_{1-\frac{W_0}{2r_0}}^{\bar{\rho}}f^2(u)u\,du}{\|f(r)\sqrt{r}\|_{L^2}^2}.
\end{align*}
Now, if $z^*< 1+\frac{W_0}{2r_0}$ then 
\begin{equation*}
1 >  \frac{\int_{1-\frac{W_T}{2z_d}}^{\bar{z}_c(\bar{\rho})}F_T^2(z)\,dz}{\|F_T(z)\|_{L^2}^2}= \frac{\int_{1-\frac{W_0}{2r_0}}^{1+\frac{W_0}{2r_0}}f^2(u)u\,du}{\|f(r)\sqrt{r}\|_{L^2}^2}=1.
\end{equation*}
If $z^*> 1+\frac{W_0}{2r_0}$ then since $\bar{z}_c(\bar{\rho})$ is a continuous monotone increasing function of $\rho$ there exists $r^*< 1+\frac{W_0}{2r_0}$ such that $z_c(r^*)=1+\frac{W_T}{2z_d}$. Consequently,
\begin{equation*}
1= \frac{\int_{1-\frac{W_T}{2z_d}}^{\bar{z}_c(r^*)}F_T^2(z)\,dz}{\|F_T(z)\|_{L^2}^2}<\frac{\int_{1-\frac{W_0}{2r_0}}^{1+\frac{W_0}{2r_0}}f^2(u)u\,du}{\|f(r)\sqrt{r}\|_{L^2}^2}=1.
\end{equation*}
Therefore it follows that $z^*=1+\frac{W_T}{2z_d}$.
\end{proof}

\subsection{Off-axis electric field}
The full electric field near the Bessel-zone $E(r,z)$ can also be approximated in the focal region $|z-z_d|\lesssim W_T/2$ using the method of stationary phase. In the dimensionless coordinates we obtain through the method of stationary phase that:
\begin{equation}\label{Stat:OffAxis}
|E(r,\bar{z})|\approx E_TF_T(\bar{z})J_0\left(\frac{\bar{k}r}{2\bar{z}r_0}\bar{\rho}_c(\bar{z})\right).
\end{equation}
That is, the beam essentially forms a Bessel-beam in the focal region. However, the function $J_0$ is also a rapidly oscillating function and the above approximation is only valid for radii satisfying $r \sim 2\bar{z}r_0/(\bar{k} \bar{\rho}_c(\bar{z}))$. This expression is not applicable outside the focal region, and of course does not cannot capture the transition of the ring beam into a Bessel-Gauss beam and back into a ring beam. 

\section{Applications}
\subsection{Remote Delivery of Bessel Beams}
In this subsection we use the method of stationary phase coupled with the Gerchberg-Saxton algorithm to obtain approximations to the optimal phase when the input profile $f(r)$ is Gaussian and the target profile $F_T(z)$ is super-Gaussian. Specifically, for $n$ a positive integer and $W_0^{\prime},W_T^{\prime}>0$ we assume that 
\begin{equation}\label{BB:functions}
f(r)=\exp\left(-\frac{(r-r_0)^2}{W_0^{\prime\, 2}}\right) \text{ and } F_T(z)=\exp\left(-\frac{(z-z_d)^{2n}}{W_T^{\prime\, 2n}}\right).
\end{equation}
Since we are using input and target functions that decay exponentially fast we can use Laplace’s method \cite{miller2006applied} to obtain accurate estimates for the normalization. Directly applying Laplace’s method it follows that
\begin{equation}\label{BB:W0Norm}
\|f(r)\sqrt{r}\|_{L^2}^2=W_0^{\prime}r_0\sqrt{\frac{\pi}{2}}+\mathcal{O}\left(\frac{W_0^{\prime\,2}}{r_0^2}\right) 
\end{equation}
and explicitly calculating we have that
\begin{equation}\label{BB:WTNorm}
\left\|F_T(z)\right\|_{L^2}^2=W_T^{\prime}\frac{2}{2^{\frac{1}{2n}}}\Gamma \left(1+\frac{1}{2n}\right).
\end{equation}
Therefore, by Equation (\ref{Form:Eqn:Normalization}) it follows that
\begin{equation}\label{BB:PeakIntensity}
\frac{E_T^2}{E_0^2} \approx \sqrt{2} \pi^{\frac{3}{2}}C(n)\frac{k r_0 W_0^{\prime}}{W_T^{\prime}},
\end{equation}
where $C(n)=2^{\frac{1}{2n}-1}\left(\Gamma \left(1+\frac{1}{2n}\right)\right)^{-1}$ is bounded between $\frac{1}{2}\sqrt{\frac{\pi}{2}}$ and $\sqrt{\frac{\pi}{2}}$. Equation (\ref{BB:PeakIntensity}) gives the ratio of the peak intensity of the on-axis profile of the beam to the peak intensity of the input beam.

The input and target functions defined above are not smooth bump functions and strictly speaking the theory and numerical algorithm we developed in the previous two sections does not apply. However, the functions given by Equation (\ref{BB:functions}) are Schwartz class and they can be very accurately approximated by non-smooth functions with compact support which in turn can be approximated by smooth functions with compact support \cite{lieb2001analysis}. In particular, we truncate these functions by forcing them to be zero for values in which the function is less than $e^{-9}\approx 10^{-4}$. This gives us the following estimates for the width of the supports
\begin{equation}\label{BB:widths}
W_0=6W_0^{\prime} \text{ and } W_T=  2\cdot 3^{\frac{1}{n}}W_T^{\prime}.
\end{equation}

With the above normalization, equations (\ref{Stat:zDifferentialEquation}) and (\ref{Stat:phaseDifferentialEquation}) can be numerically solved to construct  a guess for an optimal phase function $\phi$. Using $\phi$ to initialize the Gerchberg-Saxton algorithm, further corrections to the stationary phase construction can be obtained. To illustrate the potential applicability of our techniques to remote laser ablation, we applied the algorithm outlined in the previous section with 100 iterations of the Gerchberg-Saxton algorithm for various input and target widths with the following parameters fixed
\begin{equation}\label{BB:parameters}
\begin{array}{cccc}
z_d=1000m, & r_0=.3m, & k= 9.5\times 10^6 m^{-1}, & n=4.
\end{array}
\end{equation}
The exact profile along the optical axis is then obtained by applying the fast Fourier transform (FFT) to Equation (\ref{Form:FourierTrans}). In Figure \ref{fig:BesselTargetWidth} we plot the normalized value of $I$ for both the stationary phase analysis and the Gerchberg-Saxton algorithm for $W_0^{\prime}=.07m$ and $W_T^{\prime}$ ranging from $1m$ to $100m$. For this range of parameters, $\beta$ varies from $\beta=.4$ to $\beta=40$. The inset plot in Figure \ref{fig:BesselTargetWidth} illustrates a subset of the profiles obtained for $W_T^{\prime}=10m,\,50m,\,100m$.  In Figure \ref{fig:BesselInputWidth} we plot the normalized value of $I$ for both the stationary phase analysis and the stationary phase analysis coupled with the Gerchberg-Saxton algorithm for $W_0^{\prime}$ ranging from $.001m$ to $.1m$ and $W_T^{\prime}=20m$. For this range of parameters, $\beta$ varies from $\beta=.1$ to $\beta=10$. The inset plot in Figure  \ref{fig:BesselInputWidth} illustrates the intensity profile obtained for $W_0^{\prime}=.05m$.

\begin{figure}[htp]
        \centering
        \subfigure[][]{
                \includegraphics[width=\textwidth]{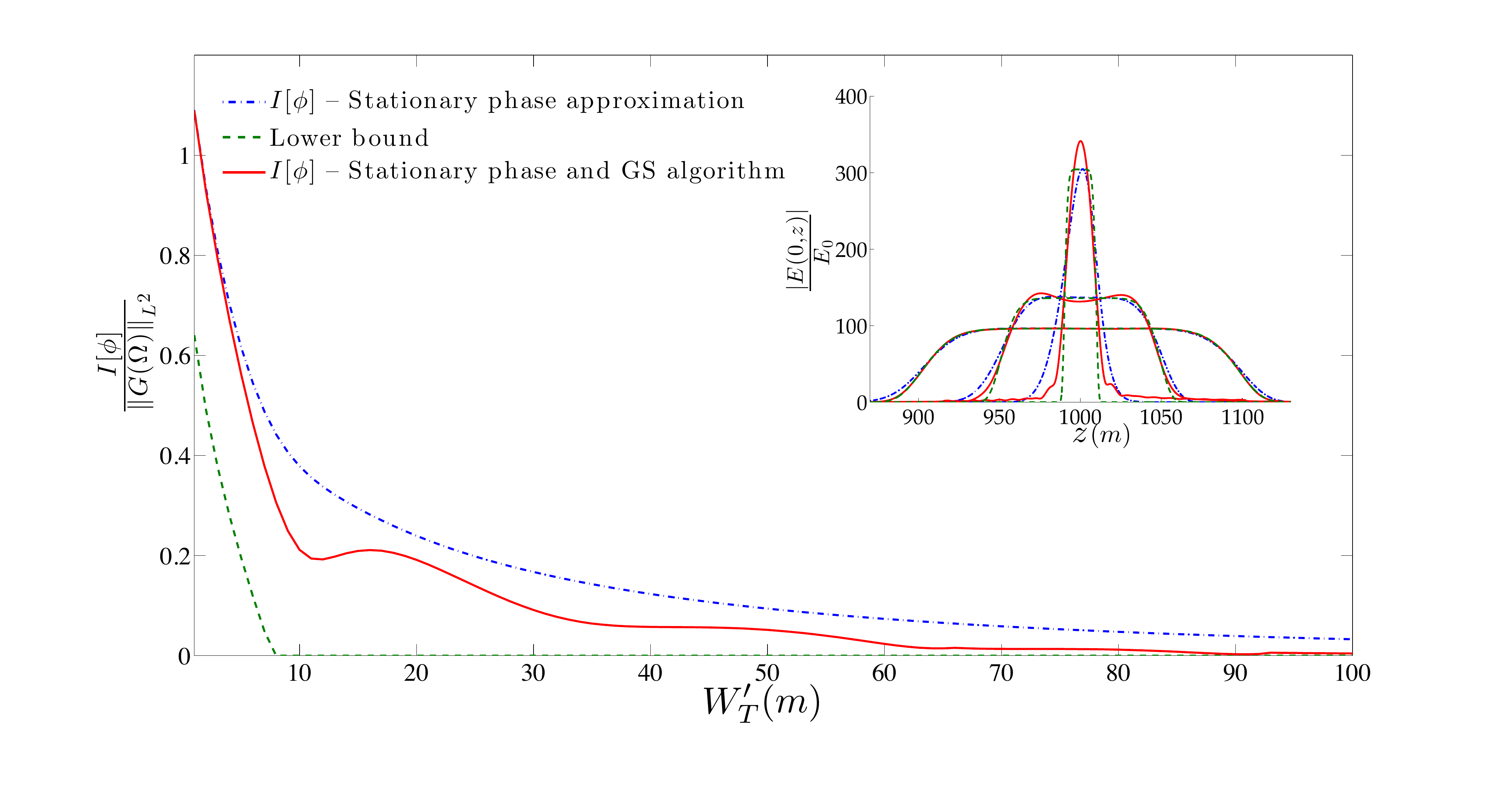}
	\label{fig:BesselTargetWidth}}
        \subfigure[][]{
                \includegraphics[width=\textwidth]{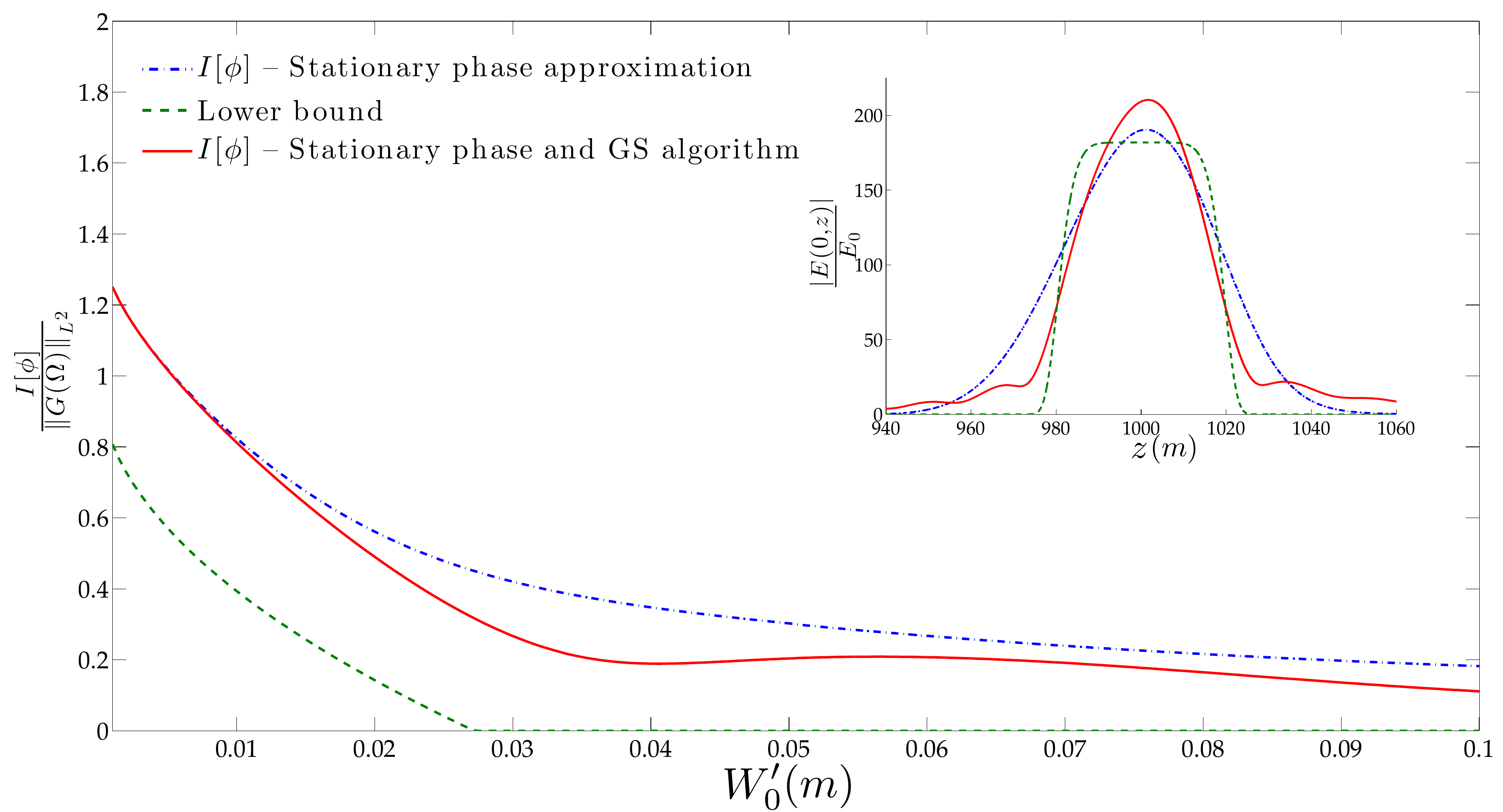}
                \label{fig:BesselInputWidth}}
                \caption{The normalized value of the functional $I$ for phase functions obtained using the stationary phase analysis as an initial guess for the Gerchberg-Saxton algorithm. The fixed parameters in both figures are $r_0=.3m$, $z_d=1000m$ and $n=4$. In the inset plots within each figure the blue dashed-dotted curves are plots of sample on-axis intensities obtained through the stationary phase analysis alone, the solid red curves are the on-axis intensities obtained by coupling the stationary phase algorithm with the Gerchberg-Saxton algorithm and the green dashed curves are the target intensity profile. \subref{fig:BesselTargetWidth}  The value of $I$ for $W_0^{\prime}=.07m$ fixed and $W_T^{\prime}$ ranging from $1m$ to $100m$. The selected intensity profiles in the inset plot  have target widths corresponding to $W_T^{\prime}=5m,$ $50.0m$ and $100m$. \subref{fig:BesselInputWidth} The value of $I$ for $W_T^{\prime}=20m$ fixed and $W_0^{\prime}$ range from $.001m$ to $.1m$. \label{fig:RemoteDelivery}}
\end{figure}

Figures  \ref{fig:BesselTargetWidth} and \ref{fig:BesselInputWidth} also illustrate the important role that $\beta$ not only plays in terms of the lower bound presented in Theorem \ref{thm:Unc1} but also the accuracy of the stationary phase approximation. For $W_T^{\prime}<8m$ $(\beta=3)$, the lower bound in Theorem \ref{thm:Unc1} is relevant and as is guaranteed by Theorem \ref{thm:Unc1} neither the stationary phase approximation nor the Gerchberg-Saxton algorithm, or indeed any other algorithm can yield accurate shaping of the beam. For $W_T^{\prime}$ ranging from $10m$ to $60m$ ($\beta=4$ to $\beta=24$) Theorem \ref{thm:Unc1} no longer applies and the stationary phase analysis yields a guess that is further improved upon by the Gerchberg-Saxton algorithm. For $W_T^{\prime}>60m$ ($\beta=24$), the stationary phase analysis yields a very accurate guess that is only slightly improved upon by the Gerchberg-Saxton algorithm.

In Figures \ref{fig:phaseplot1} and \ref{fig:phaseplot2} we plot surface and contour plots of the phase function found by coupling the stationary phase approximation with the Gerchberg-Saxton algorithm for the parameters given by Equation (\ref{BB:parameters}) with $W_0^{\prime}=.07m$ and $W_T^{\prime}$ ranging from $5m$ to $100m$. Specifically, we plot
\begin{equation}
\phi_{z_d}(r)=\phi(r)-\frac{kr^2}{2z_d}.
\end{equation}
The term $kr^2/2z_d$ corresponds to a linear shift in the Fourier variable $\Omega$ and also has an interpretation in optics as the effect of a lens that focuses a beam at the distance $z_d$ \cite{goodman2005introduction}. Interestingly, when $W_T^{\prime}=5m$ ($\beta=.4$) the phase $\phi_{z_d}$  is essentially constant. That is, except for the focusing term $kr^2/2z_d$ additional phase shaping cannot improve upon the effect of a focusing lens. However, for $W_T^{\prime}=100m$   ($\beta=40$) there is a significant difference between the numerically obtained phase and the focusing term. From Figures \ref{fig:statphaseplot1} and \ref{fig:statphaseplot2} it is again evident that the phase obtained from that Gerchberg-Saxton algorithm converges in the limit of $W_T\rightarrow \infty$ to the one obtained using the stationary phase alone.

\begin{figure}[htp]
        \centering
        \subfigure[][]{
                \includegraphics[width=.47\textwidth]{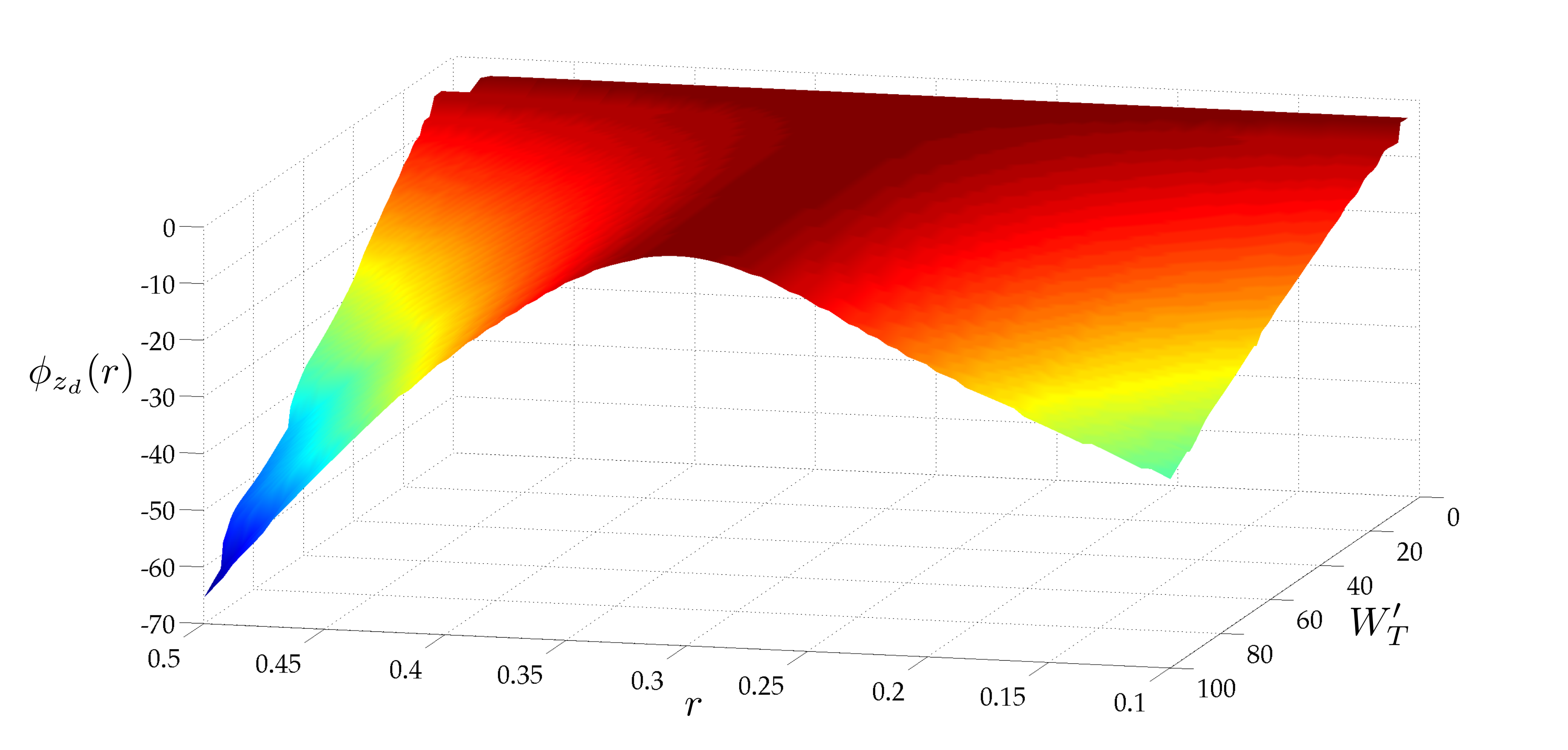}
	\label{fig:phaseplot1}}
        \subfigure[][]{
                \includegraphics[width=.47\textwidth]{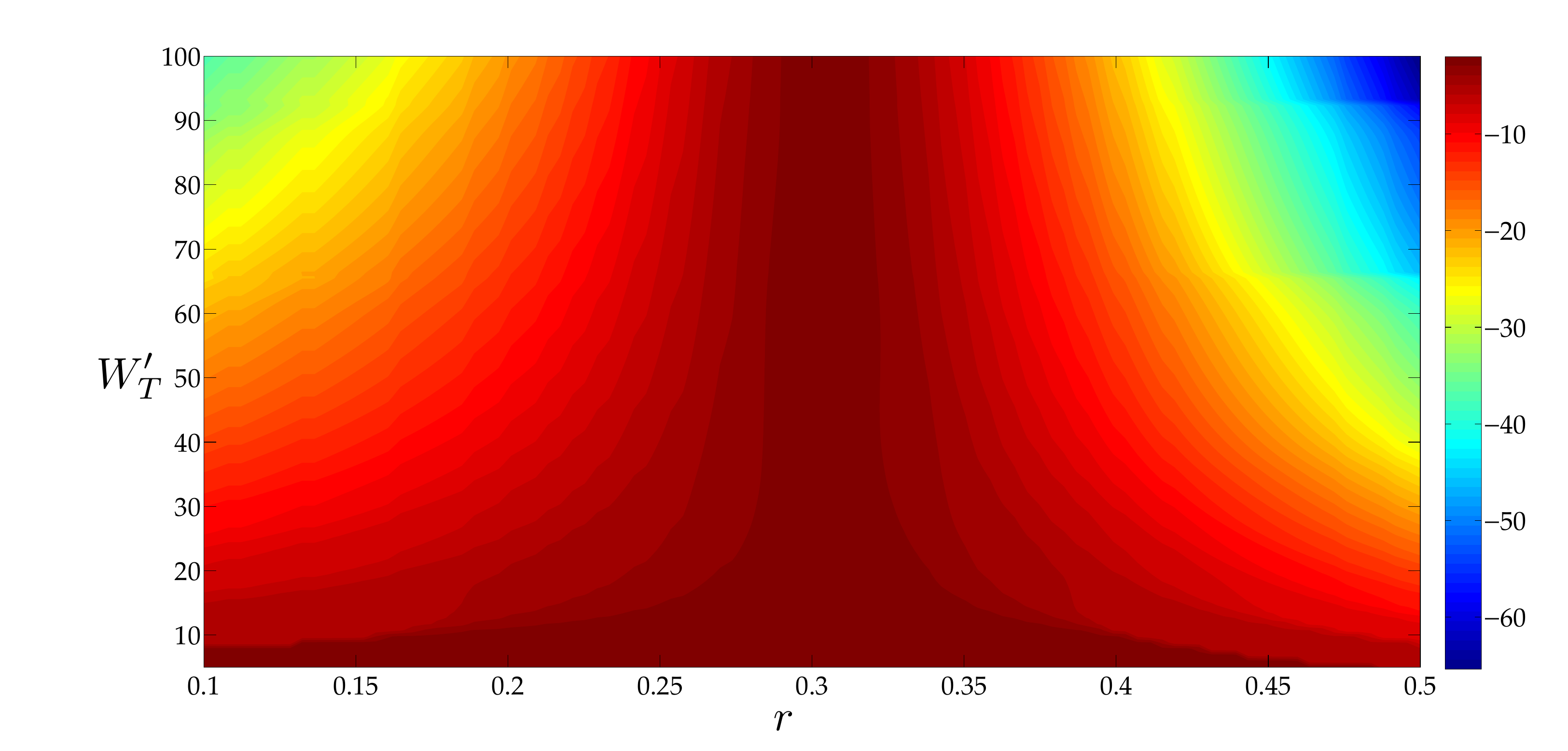}
                \label{fig:phaseplot2}}
\subfigure[][]{
                \includegraphics[width=.47\textwidth]{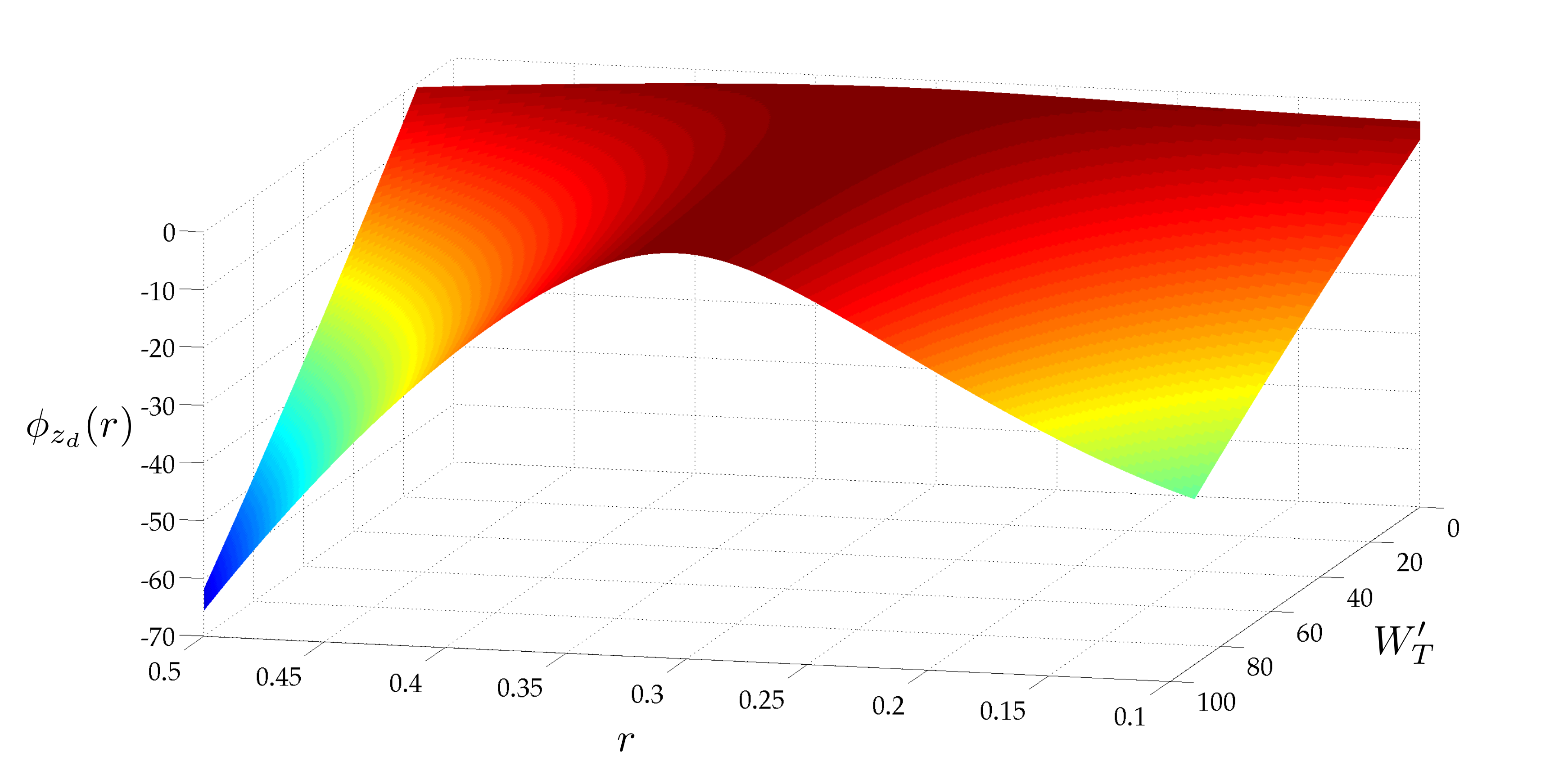}
	\label{fig:statphaseplot1}}
        \subfigure[][]{
                \includegraphics[width=.47\textwidth]{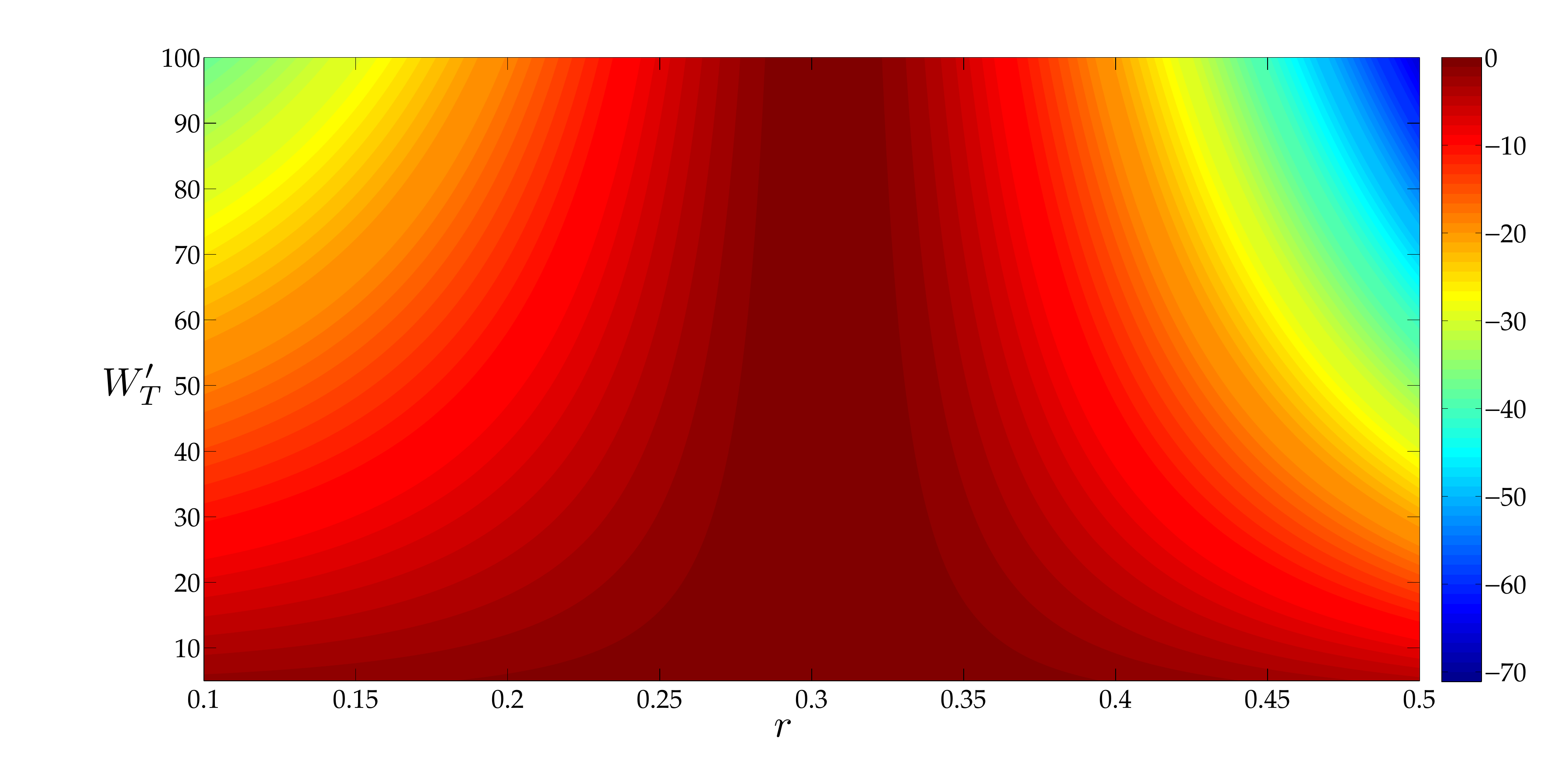}
                \label{fig:statphaseplot2}}
                \caption{\subref{fig:phaseplot1} Surface plot of the phase $\phi_{z_d}(r)=\phi(r)-kr^2/2z_d$ obtained by using the method of stationary phase to create an initial guess for $100$ iterations of the Gerchberg-Saxton algorithm with $r_0=.3m$, $W_0=.07m$, $z_d=1000m$, $n=4$ and $W_T^{\prime}$ ranging from $5m$ to $100m$. \subref{fig:phaseplot2} Contour plot of $\phi_{z_d}(r)$ for the same range of parameters. \subref{fig:statphaseplot1} Surface plot of the phase $\phi_{z_d}(r)=\phi(r)-kr^2/2z_d$ obtained by using the method of stationary phase. \subref{fig:statphaseplot2} Contour plot of $\phi_{z_d}(r)$ obtained by the method of stationary phase.}\label{fig:phaseplot}
\end{figure}

In laser ablation while it is important to have a beam with a near uniform on-axis intensity profile, the amount of power delivered by the beam is also of critical importance. In particular, numerical experiments indicate that for critical collapse of a Bessel-Gauss beam to occur the power contained in the central lobe of the beam must be near the critical power of the Townes profile \cite{fibich2000critical}.
For a fixed value of $z$ the power delivered by a beam in a circular cross section of radius $R$ about the optical axis is proportional to the following integral:
\begin{equation}
P_R(z)=\int_0^{R}|E(r,z)|^2r\,dr.
\end{equation}
Using the asymtotic formula (\ref{Stat:OffAxis})  for the off-axis profile electric field, the power in the central core of the Bessel beam at the target distance $\bar{z}=1$ can be approximated in the short wavelength limit. If we assume that $\bar{\rho}_c(1)=1$, then for $\bar{z}=1$ the electric field is approximately zero at the value $r\approx 4.96r_0/\bar{k}$. Consequently, the power delivered in the central core of the Bessel beam at the distance $z_d$ is approximated by
\begin{equation}
P_{C}(z_d)\approx E_T^2 F_T^2(z_d)\int_0^{\frac{4.96r_0}{\bar{k}}} J_0^2\left(\frac{\bar{k}r}{2r_0}\right)r dr\approx 3.12 E_T^2 F_T^2(z_d) \frac{z_d^2}{k^2 r_0^2}.
\end{equation}
From the normalization (\ref{Form:Eqn:Normalization}) it follows that in the short wavelength limit we have that:
\begin{equation*}
P_{C}(z_d)\approx 19.6 \frac{F_T^2(z_d)}{\|F_T(z)\|_{L^2}^2}\frac{z_d^2}{k r_0^2}P_0,
\end{equation*}
where $P_0$ is the initial power of the beam. Therefore, from Equation (\ref{BB:WTNorm}) it follows that
\begin{equation*}
P_C(z_d)\approx 9.8\times 2^{-\frac{1}{2n}}\Gamma\left(1+\frac{1}{2n}\right)^{-1}\frac{z_d^2}{kr_0^2 W_T^{\prime}}P_0.
\end{equation*}
This equation illustrates the trade-off between the amount of power delivered within the central lobe and the accuracy of the shaping. In particular, for $\overline{k}\gg1$ and $W_0^{\prime}\gg W_0^{\prime}$ the  shaping will be very accurate but the amount of power delivered by the beam in the central core will be much smaller than the initial power of the beam. For the range of parameters used to generate Figure \ref{fig:BesselTargetWidth}, for example, this approximation for $P_C(z_d)$ ranges from $.96P_0$ to $.01P_0$.

\subsection{Oscillatory patterns}
In this subsection we explore the possibility of using beam shaping to create intensity distributions with an oscillatory pattern for length scales that are applicable to micromachining. Again we use a Gaussian input beam:
\begin{equation}
f(r)=\exp\left(-\frac{(r-r_0)^2}{W_0^{\prime\,2}}\right)
\end{equation}
with $r_0=50mm$ and $W_0=5mm$. For integer values of of $m$ we take the target profile to be:
\begin{equation}
F_T(z)=\frac{4}{5}\exp\left(-\frac{(z-z_d)^{16}}{W_T^{\prime\, 16}}\right)\left(\cos^2\left(\frac{(z-z_d)m}{2\pi W_T^{\prime}}\right)+\frac{1}{4}\right)
\end{equation}
with $z_d=1m$, $W_T^{\prime}=1mm$ and we take $k=9.7\times 10^6m^{-1}$. As before, we can then apply the method of stationary phase coupled with the Gerchberg-Saxton algorithm to numerically optimize $I$. Now, it follows that $\beta=12.1$ for these parameters  and hence Theorem \ref{thm:Unc1} is not relevant. However, as the value of $m$ increases the target intensity develops fine scale oscillations of approximate width $w_m=2W_T^{\prime}/m$ and therefore by Corollary \ref{cor:unc1} we expect the accuracy of the beam shaping to decrease with increasing $m$. 

As in section 2 define
\begin{equation}
G(\Omega)=\frac{1}{\Omega}E_TF_T\left(\frac{k}{2\Omega}\right) \text{ and } g(s)=E_0 f\left(\sqrt{r}\right)
\end{equation}
and for $0<w_m<W_T$ define the interval $S_G(z_d,w_m)$ by
\begin{equation}
S_G(z_d,w_m) =\left\{ \Omega \in \mathbb{R}: \frac{k}{2z_d+w_m}\leq \Omega \leq \frac{k}{2z_d-w_m}\right\}.
\end{equation}
To capture the influence of the small scale oscillations on the error we define the following local quantity:
\begin{equation}
\delta_m=\frac{{\int_{S_G(z_d,W_m)}G^2(\Omega)\,d\Omega-\frac{2kw_m}{4z_d^2-w_m^2}\|g(s)\|_{L^1}}}{\int_{S_G(z_d,w_m)}G^2(\Omega)\,d\Omega}
\end{equation}
which is a local measure of the error in matching the target intensity within one period. From Corollary \ref{cor:unc1} it follows that if $\delta_m>0$ then the small scale oscillations cannot be matched by any phase function.

In Figure \ref{fig:bumps} we plot the on-axis intensity profiles for the beams obtained by using the stationary phase as an initial guess for the Gerchberg-Saxton algorithm with $100$ iterations. We selected the values $m=1,5,10,15$ to illustrate the influence of $m$ on the accuracy of the shaping. As we can see, with increasing $m$ the method of stationary phase yields a guess that matches the support of the function but fails to capture the small scale oscillations. For these values of $m$ we have that $\delta_1=-8.6$, $\delta_5=-.92$, $\delta_{10}=.04$ and $\delta_{15}=.3$ and indeed for $m=10,15$, as is to be expected, the Gerchberg-Saxton algorithm cannot match the target intensity.  It is interesting to note, however,  for large $m$ that although the Gerchberg-Saxton algorithm fails to capture the amplitude of the oscillations it does a decent job recovering the period.
\begin{figure}[ht]
 \centering
                \includegraphics[width=\textwidth]{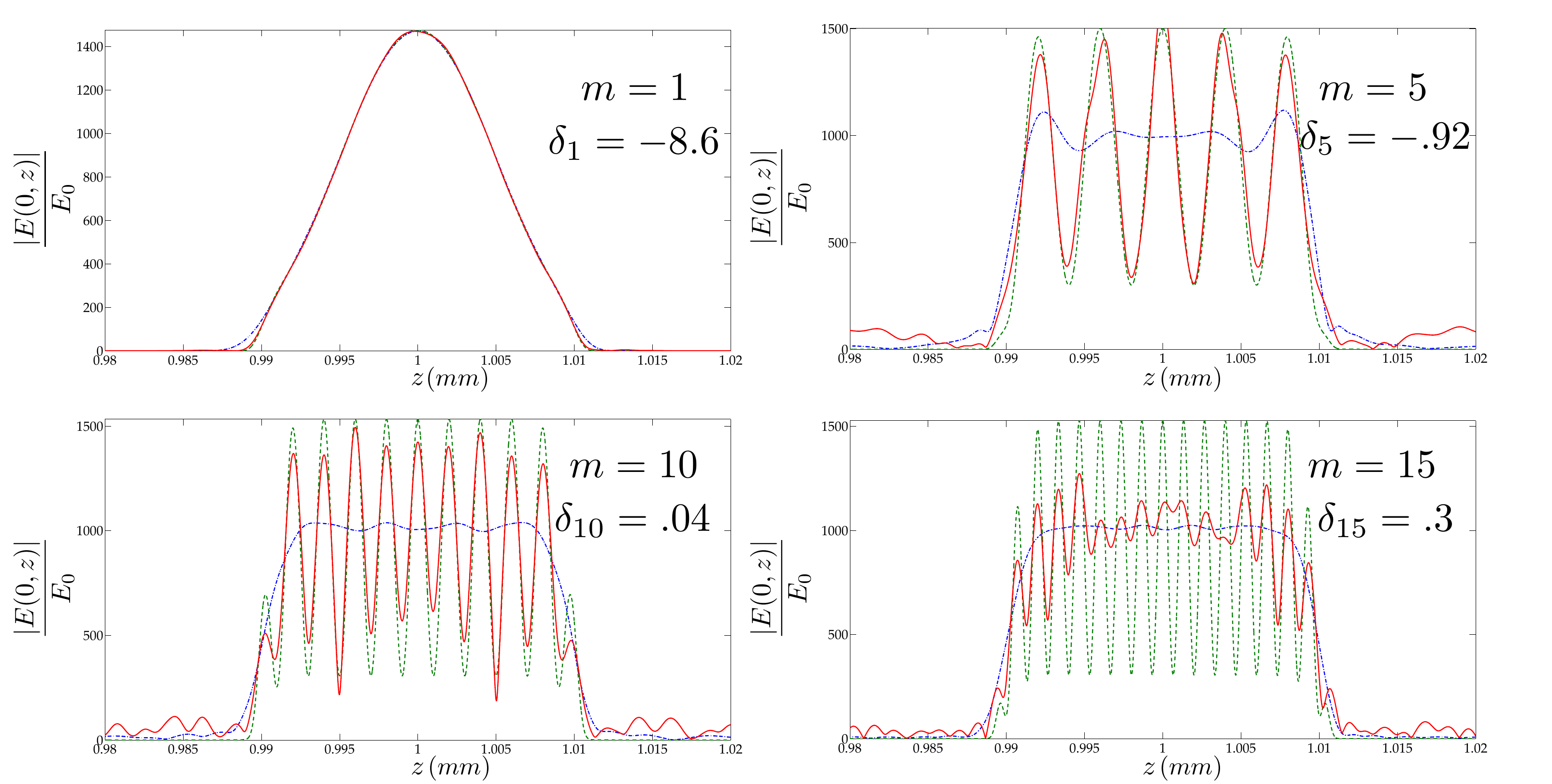}
                        \caption{Intensity profiles obtained by the stationary phase method coupled with the Gerchberg-Saxton algorithm for $m=1,5,10,15$. The blue dashed-dotted curves are the on-axis intensities obtained through the stationary phase analysis alone, the solid red curves are the on-axis intensities obtained by coupling the stationary phase algorithm with the Gerchberg-Saxton algorithm and the green dashed curves are the target intensity profile. The function $\delta_m$ quantifies the accuracy of the beam shaping and in particular for $\delta_m>0$ the amplitude of the oscillations is not matched. }\label{fig:bumps}
\end{figure}

\subsection{Remote delivery of Gaussian pulses}
In this subsection we show that for a linear medium and under the paraxial approximation the techniques we developed in this paper can be combined with linear temporal chirping to remotely deliver pulses with a desired temporal width at a target distance. The governing equation for such pulses is the temporal paraxial wave equation:
\begin{equation} \label{eqn:TempPW}
\frac{\partial E}{\partial z}=\frac{i}{2k} \Delta_{\perp} E-i \frac{\gamma}{2} \frac{\partial^2 E}{\partial t^2},
\end{equation}
where $\gamma$ is the group velocity dispersion for a pulse with wavenumber $k$ \cite{newell1992nonlinear}. If we assume separable initial data of the form $E(r,0,t)=H(r)T(t)$ then the solution to Equation (\ref{eqn:TempPW}) can be expressed in terms of the Hankel transform of a convolution:
\begin{align}
E(r,z,t)=&\frac{-ik}{z^{3/2}\sqrt{\pi \gamma}\sqrt{1-i}}\int_0^{\infty}\int_{-\infty}^{\infty}H(r)\exp\left(\frac{i k(r^2+\rho^2)}{2z}\right)J_0\left(\frac{k \rho r}{z}\right)\nonumber\\
&\,\times T(t^{\prime})\exp\left(-i\frac{(t-t^{\prime})^2}{2z\gamma}\right)\rho d\rho dt^{\prime}.
\end{align}

Again, we assume that the input spatial profile is given by $H(r)=E_0f(r)\exp(i \phi(r))$, where $E_0>0$, $f\in C_0^{\infty}(\mathbb{R}^+)$ is supported on the interval $r_0-W_0/2<r<r_0+W_0/2$ and $\phi$ is a measurable function corresponding to a shaper phase. We assume that the input temporal profile is a linearly chirped Gaussian pulse:
\begin{equation}
T(t)=\exp\left(-\frac{t^2}{\tau_0^2}\right)\exp\left( i \frac{t^2}{\alpha^2}\right)
\end{equation}
where $\tau_0>0$ is a measure of the temporal width of the initial pulse and $\alpha>0$ is a linear chirping parameter. With this initial data, the on-axis intensity is given by
\begin{equation}
\left|E(0,z,t)\right|=\frac{k}{z\left(q(z)\right)^{\frac{1}{4}}}\exp\left(-\frac{t^2}{\tau_0^2 q\left(z\right)}\right)\left|\int_0^{\infty} E_0f(\rho)\exp\left(i\phi(\rho)\right)\exp\left(i \frac{\rho^2}{2z}\right)\rho d\rho\right|,
\end{equation}
with $q(z)$  a quadratic function defined by
\begin{equation}
q(z)=\left(1-\frac{2z\gamma}{\alpha^2}\right)^2+\frac{4z^2\gamma^2}{\tau_0^4}.
\end{equation}
The function $q(z)$ quantifies how the temporal width broadens or contracts as it propagates; indeed the characteristic temporal width is $W(z)=\tau_0 \sqrt{q(z)}$.

We assume that the target profile along the optical axis is also given by a Gaussian pulse of the form
\begin{equation}
E_T(z,t)=E_T F_T(z)\exp\left(-\frac{t^2}{\tau_T^2}\right),
\end{equation}
where $E_T>0$, $F_T\in C_0^{\infty}(\mathbb{R}^+)$ with support $z_d-W_T/2<z< z_d+W_T/2$ and $\tau_T>0$. To ensure that $W(z)$ matches the desired temporal width at $z_d$ and the change in $W(z)$ over the support of $F_T$ is as small as possible we must have that
\begin{equation*}
W(z_d)=\tau_T \text{ and } \left.\frac{d W}{dz}\right|_{z_d}=0.
\end{equation*}
Solving these equations yields the conditions:
\begin{equation}\label{Pulse:parameters}
\alpha^2=\frac{ \tau_T^4 +4z_d^2 \gamma^2}{2z_d \gamma} \text{ and } \tau_0^2=\frac{\tau_T^4 +4z_d^2 \gamma^2}{\tau_T^2}.
\end{equation}

Define $G(\Omega)$ by 
\begin{equation}
G(\Omega)=\frac{k\left(q\left(\frac{k}{2\Omega}\right)\right)^{\frac{1}{4}}}{\Omega}F_T\left(\frac{k}{2\Omega}\right)
\end{equation}
and as we did before define $g(s)=E_0f(\sqrt{s})$, $ \varphi(s)=\phi(\sqrt{s})$. The problem of designing the optimal shaper phase with this temporal chirping then corresponds to minimizing the functional $I:\mathcal{M}\mapsto \mathbb{R}^+$ defined by $I[\varphi]=\|G-|\mathcal{F}(g \exp(i \varphi))| \|_{L^2}$ which problem corresponds to the variational problem we considered before except for the additional term
\begin{equation*}
k\left(q\left(\frac{k}{2\Omega}\right)\right)^{\frac{1}{4}}\Omega^{-1}
\end{equation*}
accounting for the additional modification to the spatial intensity arising from temporal broadening. 

In Figure \ref{fig:temporalprofile} we plot the on-axis spatial-temporal profile for a Gaussian pulse propagating in air obtained by using the method of stationary phase. For this figure we again assumed that 
\begin{equation}
f(r)=\exp\left(-\frac{(r-r_0)^2}{W_0^{\prime\, 2}}\right) \text{ and } F_T(z)=\exp\left(-\frac{(z-z_d)^{2n}}{W_T^{\prime\, 2n}}\right),
\end{equation}
with 
\begin{equation*}
\begin{array}{cccc}
\displaystyle{z_d=1000m}, & \displaystyle{r_0=.5m}, & k= 7.9\times 10^6 m^{-1}, & \displaystyle{n=8}
\end{array}
\end{equation*}
\begin{equation*}
\begin{array}{cc}
 \displaystyle{W_0^{\prime}=.1m}, & W_T^{\prime}=20m.
\end{array}
\end{equation*}
Following the experimental results in \cite{wrzesinski2011group} we took the group velocity dispersion of air for this wavelength of light to be $\gamma=20\text{(fs)}^2m^{-1}$ and we assumed a target temporal width of $\tau_T=50$fs. For these values the initial temporal width is $\tau_0\approx 800$fs and the chirping parameter is $\alpha= 200$fs.
\begin{figure}[ht]
        \centering
                \includegraphics[width=\textwidth]{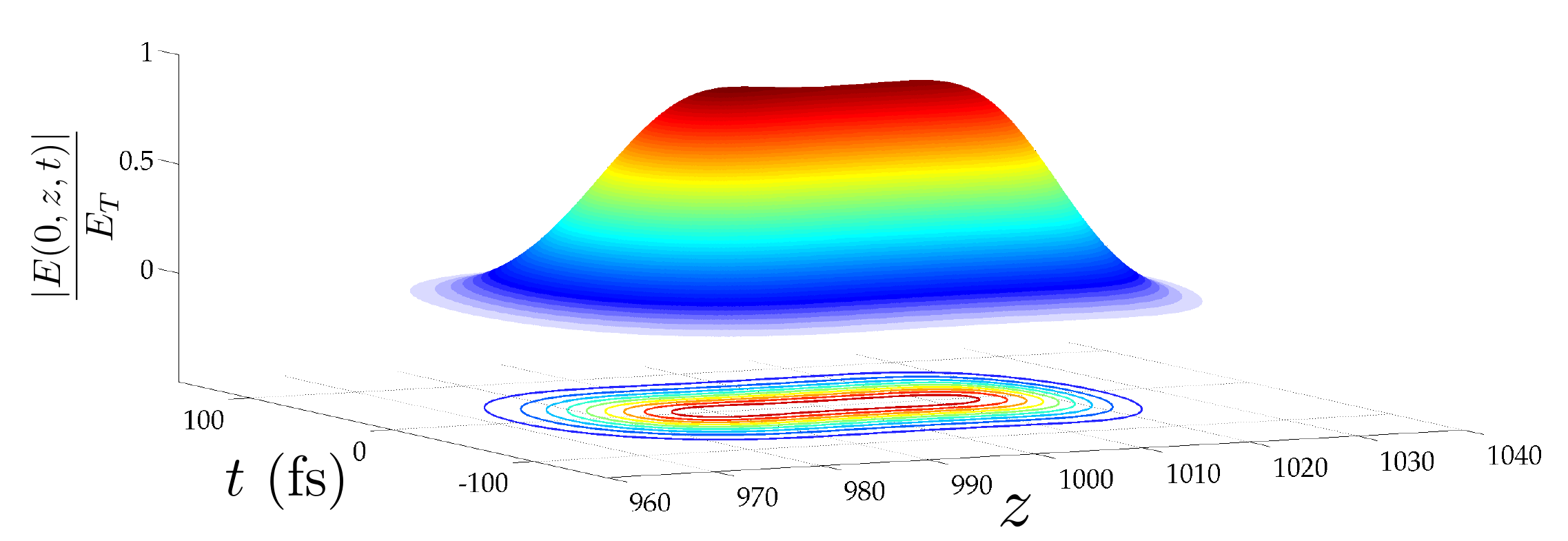}
                        \caption{Spatial temporal profile of a Gaussian pulse with near uniform on-axis intensity propagating in the atmosphere. The profile was obtained using the method of stationary phase.}\label{fig:temporalprofile}
\end{figure}

\section{Summary and discussion}
In this paper we studied the problem of phase shaping an annular beam into a desired intensity profile along the optical axis by analyzing minimizers of the functional $I: \mathcal{M} \mapsto \mathbb{R}$ defined by Equation \ref{Form:DefFunc}. Formally, this functional is equivalent to the ones arising in phase retrieval from two intensity measurements. Typically, due to its speed and ease of implementation, the Gerchberg-Saxton algorithm is used to construct approximate minimizers. While the Gerchberg-Saxton algorithm is an error reducing algorithm there are no known general convergence results for this algorithm. In particular the algorithm typically stagnates away from the global minimum which in phase retrieval can be overcome by additional modifications to the algorithm such as the hybrid input-output method \cite{Fienup:1982di} or those in \cite{pasienski2008high}. Moreover, in phase retrieval the infimum of $I$ is zero and convergence of the Gerchberg-Saxton algorithm and its modifications can be assessed by the proximity of $I$ to zero. However, as a consequence of the uncertainty principle the minimum value of $I$ may be significantly bounded away from zero. Consequently it is not clear if modifications to the Gerchberg-Saxton algorithm are applicable to beam shaping problems nor is it clear what a sufficient convergence criterion of such an algorithm would be.

The central result of our work is the identification of the dimensionless quantity
\begin{equation*}
\beta=\frac{2k r_0 W_T W_0}{4z_d^2-W_T^2}
\end{equation*}
and the associated lower bounds on the functional $I$. We identified three scaling regimes for $\beta$ for which different quantitative and qualitative behaviors of the optimal profile occur. These regimes are as follows.
\begin{enumerate}
\item For $\beta\ll \pi$ the uncertainty principle guarantees that the optimal profile will differ significantly from the target profile. This result has the consequence that neither the Gerchberg-Saxton algorithm nor any other numerical algorithm will yield accurate shaping of the beam. 

\item For $\beta \gg \pi$ the method of stationary phase yields an approximation to the problem that is very close to the optimal profile. This asymptotic regime can be considered within the geometrical optics setting in the sense that the light rays originating from the input plane can be accurately mapped to the target intensity profile through the method of stationary phase. Indeed, for the problem of creating a beam with a nearly uniform on-axis intensity, the Gerchberg-Saxton algorithm coupled with the method of stationary phase yields nearly optimal profiles.  In particular, the method of stationary phase can be used as a robust method for overcoming the stagnation issues of the Gerchberg-Saxton algorithm commented on in \cite{durfee2013phase}. 

 \item For the intermediate regime where $\beta\sim \pi$ the phase produced by the method of stationary phase is significantly improved upon by the Gerchberg-Saxton algorithm. However, in this asymptotic regime the uncertainty principle no longer applies. Moreover, the value of $I$ for $\phi$ obtained through the method of stationary phase is not close to zero. Consequently, within this regime there is no clear convergence of the algorithm. Following the reasoning in Remark \ref{form:Remark} we could obtain more accurate lower bounds on the functional that account for the specific forms of the input and target intensity profile instead of simply the widths of their supports. However, in the regime $\beta \sim \pi$ a universal scaling law for the error in terms of the wavelength is not possible. 
\end{enumerate}

%We also proved in this paper the existence of a minimizer for this problem and studied the regularity of the optimal phase. While for phase retrieval it is clear that a minimum exists this was not clear at the outset of our study. From the existence of a minimum for this problem we also proved that as a consequence of the Paley-Wiener Theorem that the optimal phase is smooth except for possible jump discontinuities of $2\pi$. In particular, since we are initializing the Gerchberg-Saxton algorithm with a smooth phase obtained from the stationary phase analysis this result implies that the numerically optimal phase will be smooth. This regularity result is important for the fabrication of optical elements that will achieve the phase shaping.

The functional we considered was derived under the paraxial assumption and hence is applicable for linear beams with the same range of wavelengths and geometrical length scales as in the Fresnel approximation. Indeed, for the range parameters we considered in the applications, the Fresnel number ranges from $150$ to $4000$. Consequently, for these parameters the integral transform given by Equation (\ref{Intro:OpticalAxis}) very accurately approximates the propagation of the on-axis electric field for linear beams. Therefore, the results of this paper and in particular the uncertainty principle are fundamental and place restrictions on the types of intensity profiles that can be achieved through linear phase shaping alone. 

For specific applications and in particular those requiring high power beams it may be necessary to consider the off-axis electric field and nonlinear interactions. Equation (\ref{Stat:OffAxis}) gives an asymptotic result indicating that the intensity pattern near the target distance will form a Bessel-like function. Numerical evidence indicates that for critical collapse of the beam to occur the power in the central core of the Bessel function must be near that of the Townes profile \cite{fibich2000critical}. For the applications we considered, it follows from Equation (\ref{Form:Eqn:Normalization}) that the initial peak intensity of the ring beam $E_0^2$ is much smaller than the target peak intensity $E_T^2$ of the generated Bessel beam and hence Kerr nonlinearities will be negligible until the ring beam is focused into the Bessel-like pattern. Predicting the filamentation patterns of a beam after critical collapse is a challenging problem and perhaps the nonlinear geometrical optics method presented in \cite{gavish2008predicting} could be used in conjunction with our methods near the target distance. The full effect of the Kerr nonlinearities will be addressed in subsequent publications. Additionally, for propagation of the beam over ranges of $1km$ it will be necessary to consider the effects of turbulence in the atmosphere which is beyond the scope of this paper. \\
\appendix

\section{Radially symmetric solutions to the paraxial wave equation} \label{Appendix:Solution}
Consider the initial value problem with radially symmetric initial data:
\begin{equation}
\begin{cases}
\displaystyle{\frac{\partial E}{\partial z}=\frac{i}{2k} \Delta_{\perp}E}\\
\displaystyle{E(r,0)=E_0f(r)} 
\end{cases},
\end{equation}
where $f$ is a smooth function with compact support. If we let $\mathcal{H}[g](k_{\perp})$ and $\mathcal{H}^{-1}[G](r)$ denote the Hankel transform and its inverse defined by
\begin{equation}
\mathcal{H}[g](k_{\perp})=\int_0^{\infty} g(r)J_0(rk_{\perp})\,dr \text{ and } \mathcal{H}^{-1}[G](r)=\int_0^{\infty} G(k_{\perp})J_0(rk_{\perp})k_{\perp}dk_{\perp} 
\end{equation}
then the solution to the initial value problem is given by
\begin{equation}
E(r,z)=\int_0^{\infty} E_0\mathcal{H}[f(r)](k_{\perp}) \exp\left(-\frac{i k_{\perp}^2 z}{2k} \right)J_0\left(rk_{\perp}\right)k_{\perp} dk_{\perp}.
\end{equation}
Noting that 
\begin{equation}
\mathcal{H}^{-1}\left[\exp\left(-\frac{i k_{\perp}^2}{2k}\right)\right](r)=-\frac{ik}{z} \exp\left( \frac{i k r^2}{2z}\right)
\end{equation}
it follows from the convolution theorem for Hankel transforms \cite{poularikas2009transforms} that 
\begin{align}
E(r,z)&=-\frac{i k}{2\pi z} \int_0^{2\pi}\int_0^{\infty} f(\rho) e^{\frac{ik}{2z}\left(r^2-2r\rho(\cos(\theta) \cos(\phi)+\sin(\theta)\sin(\phi))+\rho^2\right)}\rho\, d\rho d \phi.\\
&=-\frac{ik}{z}\int_0^{\infty} E_0 f(\rho)\exp(i \phi(\rho))\exp\left(\frac{ i k (r^2+\rho^2)}{2z}\right)J_0\left( \frac{ k \rho r}{z}\right) \rho\, d\rho.
\end{align}

\noindent \textbf{Acknowledgments:} The authors acknowledge funding support under the MURI AFOSR grant FA9550-10-0561. C.D. acknowledges funding support from AFOSR under the grant FA9550-10-1-0394 an S.V. acknowledges support from the NSF grant 0807501. The authors wish to thank Ewan Wright for many useful discussions and bringing to our attention reference  \cite{Bagini} and we would like to thank Colm Dineen who constructed Figure \ref{fig:Besselsubfig1}. The authors would also like to thank the two anonymous referees who reviewed this paper. Their comments greatly improved the organization and the presentation of our results. \\

\noindent \textbf{References}
\bibliographystyle{elsarticle-num}
%\bibliography{tau}
\def\cprime{$'$}

\end{document}